\documentclass[11pt]{article}

\usepackage{graphicx, amssymb, latexsym, amsfonts, amsmath, lscape, amscd,
amsthm, color, epsfig, mathrsfs, tikz, enumerate}

\newtheorem{theorem}{Theorem}[section]

\newtheorem{lemma}[theorem]{Lemma}

\newtheorem{question}[theorem]{Question}

\newcommand\DELETE[1]{}

\newcommand\ESOK[1]{#1}

\begin{document}

\title{{\bf Outerplanar and planar oriented cliques}}
\author{
{\sc Ayan Nandy}$\,^a$, {\sc Sagnik Sen}$\,^{a,b,c}$, {\sc \'Eric Sopena}$^{b,c}$\\
\mbox{}\\
{\small $(a)$ Indian Statistical Institute, Kolkata, India}\\
{\small $(b)$ Univ. Bordeaux, LaBRI, UMR 5800, F-33400 Talence, France}\\
{\small $(c)$ CNRS, LaBRI, UMR 5800, F-33400 Talence, France}
}

\date{\today}

\maketitle

\begin{abstract}
\DELETE{An oriented graph is a directed graph without cycles of length at most 2. 
An \textit{oriented $k$-coloring} of an oriented graph $ \overrightarrow{G}$ is a mapping $\phi$ from the vertex set $V(\overrightarrow{G})$ to the set $\{1,2,....,k\}$  such that, (i) $\phi(u) \neq \phi(v)$ whenever $u$ and $v$ are adjacent and (ii) if $ \overrightarrow{uv}$ and $ \overrightarrow{wx}$ are 
two arcs in $\overrightarrow{G}$, then $\phi(u) = \phi(x)$ implies $\phi(v) \neq \phi(w)$. 
 The \textit{oriented chromatic number} $\chi_o(\overrightarrow{G})$ of an oriented graph $\overrightarrow{G}$ is the smallest integer $k$ for which 
$ \overrightarrow{G}$ has an oriented $k$-coloring.
An \textit{oclique} 
 is an oriented graph $ \overrightarrow{G}$ such that $\chi_o(\overrightarrow{G}) = |V(\overrightarrow{G})|$. 
 It is NP-hard to decide if a graph can be oriented as an oclique or not. 
}

\ESOK{The clique number of an undirected graph $G$ is the maximum order of a complete subgraph of $G$ and is
a well-known lower bound for the chromatic number of $G$. Every proper $k$-coloring of $G$ may be viewed as
a homomorphism (an edge-preserving vertex mapping) of $G$ to the complete graph of order $k$. By considering
homomorphisms of oriented graphs (digraphs without cycles of length at most 2), we get a natural notion of (oriented) colorings
and oriented chromatic number of oriented graphs. 
An oriented clique is then an oriented graph whose number of vertices and oriented chromatic number coincide.
However, the structure of oriented cliques is much less understood than in the undirected case.}

 \ESOK{In this paper, we study the structure of outerplanar and planar oriented cliques.
 We first provide} a list of 11 graphs and prove that an outerplanar graph can be oriented as an oriented clique if and only if 
 it contains one of \ESOK{these graphs} 
 as a spanning subgraph. 
Klostermeyer and MacGillivray conjectured that \ESOK{the order of a} planar oriented clique \ESOK{is} at most 15, which was later proved by Sen.
We show that 
any planar oriented clique on 15 vertices must contain a particular oriented graph as a spanning \ESOK{subgraph, thus reproving the} above conjecture.
We also provide tight upper bounds 
for the order of planar oriented cliques of girth $k$ for all $k \ge 4$.
 
\end{abstract}


\section{Introduction \ESOK{and statement of results}}

An {\textit{oriented graph}} 
 is a \ESOK{digraph} with no cycle of length 1 or 2. By replacing each edge of a simple graph $G$ with an arc (ordered pair of vertices)
  we obtain an oriented graph $ \overrightarrow{G}$;  \ESOK{we say that} $ \overrightarrow{G}$ is  an \textit{orientation} of $G$ and \ESOK{that}
  $G$ is the \textit{underlying graph} of $ \overrightarrow{G}$.
  We denote by $V( \overrightarrow{G})$  and $A( \overrightarrow{G}$)  the set of vertices  and arcs  of $ \overrightarrow{G}$, \ESOK{respectively}.  
   An arc $(u,v)$ (where $u$ and $v$ are vertices) is denoted by $ \overrightarrow{uv}$. Two arcs $\overrightarrow{uv}$ and $\overrightarrow{vw}$ 
  of an oriented graph are together called a \textit{directed 2-path}, or a \textit{2-dipath}, where $u$ and $w$ are \textit{terminal} vertices and $v$ is an 
  \textit{internal} vertex.

Colorings of oriented graphs first appeared in the work of Courcelle~\cite{courcelle-monadic} on the monadic second order logic of graphs. 
Since then it has been considered by many researchers, following the work of Raspaud and Sopena~\cite{planar80} on oriented colorings
of planar graphs.

An \textit{oriented $k$-coloring}~\cite{ocoloring} of an oriented graph $ \overrightarrow{G}$ is a mapping $\phi$ from 
$V(\overrightarrow{G})$ to the set $\{1,2,....,k\}$  such that:

\begin{itemize}
\item[$(i)$]  $\phi(u) \neq \phi(v)$ whenever $u$ and $v$ are adjacent, and 

\item[$(ii)$]  if $ \overrightarrow{uv}$ and $ \overrightarrow{wx}$ are 
two arcs in $\overrightarrow{G}$, then $\phi(u) = \phi(x)$ implies $\phi(v) \neq \phi(w)$. 
\end{itemize}

\ESOK{We  say} that an oriented graph $\overrightarrow{G}$ is 
\textit{$k$-colorable} whenever it admits an oriented $k$-coloring. 
 The \textit{oriented chromatic number} $\chi_o(\overrightarrow{G})$ of an oriented graph $\overrightarrow{G}$ is the smallest integer $k$ such that 
$ \overrightarrow{G}$ is $k$-colorable.

Alternatively, one can define \ESOK{the} oriented chromatic number by \ESOK{means of} homomorphisms of oriented graphs. 
Let $ \overrightarrow{G}$ and $ \overrightarrow{H}$ be two oriented graphs. A homomorphism of $ \overrightarrow{G}$ to $ \overrightarrow{H}$ is a
 mapping $\phi: V(\overrightarrow{G}) \rightarrow V(\overrightarrow{H})$ which preserves the arcs, that is, $uv \in A(\overrightarrow{G})$ implies $\phi(u)\phi(v) \in A(\overrightarrow{H})$. 
 The \textit{oriented chromatic number} $\chi_o( \overrightarrow{G})$ of an oriented graph $\overrightarrow{G}$ is 
 then the minimum order (number of vertices) of an oriented graph $\overrightarrow{H}$ such that
$\overrightarrow{G}$
admits a homomorphism to $\overrightarrow{H}$.

Notice that the terminal vertices of a  
\ESOK{2-dipath} must receive  distinct colors in every oriented coloring because of the second condition of the definition. 
In fact, for providing an oriented coloring of an oriented graph, only the pairs of vertices which are either adjacent or connected by a 2-dipath must receive distinct colors 
(that is, \ESOK{for every two non-adjacent vertices $u$ and $v$ which are not linked by a 2-dipath, there exists an oriented coloring which assigns the same color to $u$ and $v$}).
Motivated by this observation, the following definition
was proposed.

An \textit{\ESOK{absolute oriented} clique}, or simply an \textit{oclique} --- a term coined by Klostermeyer and MacGillivray in~\cite{36} ---, is an oriented graph  $ \overrightarrow{G}$ for which $\chi_o(\overrightarrow{G}) = |V(\overrightarrow{G})|$.    
  Note that ocliques can hence be characterized as those oriented graphs whose any two distinct vertices are 
  at (weak) directed distance at most 2 from each other,
  that is, either adjacent or connected by a 2-dipath in either direction. 
 Note that an oriented graph with an oclique of order  $n$ as a subgraph has oriented chromatic number  at least $n$. 
 The \textit{\ESOK{absolute oriented} clique number} $\omega_{ao}(\overrightarrow{G})$ of an oriented graph $ \overrightarrow{G}$ is the maximum order of an oclique  contained in 
 $ \overrightarrow{G}$ as a subgraph.

The oriented chromatic number $\chi_{o}(G)$ (resp. \ESOK{absolute oriented} clique number $\omega_{ao}(G)$)  of  a simple graph 
$ G$  is the maximum of the oriented chromatic numbers (resp. \ESOK{absolute oriented} clique numbers) of all the oriented graphs with underlying graph  $G$. 
The oriented chromatic number $\chi_{o}(\mathcal{F})$ (resp. \ESOK{absolute oriented} clique number $\omega_{ao}(\mathcal{F})$) of  a family 
$ \mathcal{F}$ of graphs is the maximum of the oriented chromatic numbers (resp. \ESOK{absolute oriented} clique numbers) of the graphs from the family $\mathcal{F}$.

 From the definitions, clearly we have the following:

\begin{lemma}\label{lem basic-in}
For any oriented graph $\overrightarrow{G}$, 
$ \omega_{ao}(\overrightarrow{G})  \leq \chi_o(\overrightarrow{G})$.
\end{lemma}

One of the first major results proved regarding \ESOK{the} oriented chromatic number of planar graphs is the following by Raspaud and Sopena~\cite{planar80}.

\begin{theorem}[Raspaud and Sopena, 1994]\label{orientedacyclic}
Every planar graph has oriented chromatic number at most $80$. 
\end{theorem}

In the same paper, they also proved that every oriented forest is 3-colorable. 

\begin{theorem}[Raspaud and Sopena, 1994]\label{orientedforest}
Every forest has oriented chromatic number at most~$3$. 
\end{theorem}

Later, Sopena~\cite{orientedchi} proved that every oriented outerplanar graph is 7-colorable and provided an example of an outerplanar oclique of order 7 (Figure~\ref{fig outerplanarmax}) to prove the tightness of the result. 

\begin{theorem}[Sopena, 1997]\label{orientedouterplanar}
Every outerplanar graph has oriented chromatic number at most~$7$. 
\end{theorem}

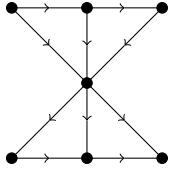
\begin{figure}

\centering
\begin{tikzpicture}

\filldraw [black] (0,0) circle (2pt) {node[below]{}};
\filldraw [black] (1,0) circle (2pt) {node[below]{}};
\filldraw [black] (2,0) circle (2pt) {node[below]{}};

\filldraw [black] (1,1) circle (2pt) {node[left]{}};

\filldraw [black] (0,2) circle (2pt) {node[above]{}};
\filldraw [black] (1,2) circle (2pt) {node[above]{}};
\filldraw [black] (2,2) circle (2pt) {node[above]{}};

\draw[->] (0,2) -- (.5,1.5);
\draw[-] (.5,1.5) -- (1,1);

\draw[->] (1,2) -- (1,1.5);
\draw[-] (1,1.5) -- (1,1);

\draw[->] (2,2) -- (1.5,1.5);
\draw[-] (1.5,1.5) -- (1,1);

\draw[->] (1,1) -- (.5,.5);
\draw[-] (.5,.5) -- (0,0);

\draw[->] (1,1) -- (1,.5);
\draw[-] (1,.5) -- (1,0);

\draw[->] (1,1) -- (1.5,.5);
\draw[-] (1.5,.5) -- (2,0);

\draw[->] (0,0) -- (.5,0);
\draw[-] (.5,0) -- (1,0);

\draw[->] (1,0) -- (1.5,0);
\draw[-] (1.5,0) -- (2,0);

\draw[->] (0,2) -- (.5,2);
\draw[-] (.5,2) -- (1,2);

\draw[->] (1,2) -- (1.5,2);
\draw[-] (1.5,2) -- (2,2);


\end{tikzpicture}

\caption{\ESOK{The outerplanar} oclique $\vec{O}$ of order 7.}\label{fig outerplanarmax}
\end{figure}

\ESOK{The structure of ocliques is much less understood than in the undirected case, where a clique is nothing but a complete graph.
For instance, the exact value of the minimum number of arcs in an oclique of order $k$ is not known yet. F\"uredi, Horak, Pareek and Zhu~\cite{FHPZ98}, and Kostochka, 
\L uczak, Simonyi and Sopena~\cite{KLSS99},
independently proved that this number is $(1+o(1))k\log_2 k$.
}

The questions related to the \ESOK{absolute oriented} clique number \ESOK{of planar graphs} have been first asked by Klostermeyer and MacGillivray~\cite{36} in 2002. 
In their paper they asked: ``what is the maximum order of a planar oclique?'', which is equivalent to 
asking ``what is the  \ESOK{absolute oriented} clique number of planar graphs?''.
 In order to find the answer to this question, Sopena~\cite{16} found a planar oclique of order 15 (Figure~\ref{fig planarmax}) while  Klostermeyer and MacGillivray~\cite{36} showed that there is no planar oclique of 
order more than 36, improving the upper bound of 80 which can be obtained by using Lemma~\ref{lem basic-in} and Theorem~\ref{orientedacyclic},  and conjectured that the maximum order of a planar oclique  is 15. Later in 2011, the conjecture was positively settled~\cite{sen} 
and we will state it as Theorem~\ref{orientedplanarabsolute}(a) in this article.

Klostermeyer and MacGillivray also showed that any outerplanar oclique of order 7 must contain a particular unique spanning subgraph (oriented).

\begin{theorem}[Klostermeyer and MacGillivray, 2002]~\label{gary-outer} 
An oriented outerplanar graph of order at least 7 is an oclique  if and only if it contains \ESOK{the outerplanar oclique $\overrightarrow{O}$ depicted in Figure~\ref{fig outerplanarmax}} as a spanning subgraph.
\end{theorem}

Bensmail, Duvignau and Kirgizov~\cite{oclique-complexity} showed that given an undirected graph $G$ it is NP-hard to decide if $G$ has an orientation $\overrightarrow{G}$
such that $\overrightarrow{G}$ is an oclique (the similar problem, but using the directed distance instead of the {\em weak} directed distance, was shown to be also NP-complete by 
Chv\'atal and Thomassen in~\cite{CT}). Now  it is easy to notice from Theorem~\ref{gary-outer} that an undirected outerplanar graph of order
at least 7 can be oriented as an oclique if and only if it contains \ESOK{the graph} $O$ (\ESOK{the} underlying undirected graph 
of \ESOK{the oriented graph} $\overrightarrow{O}$ from Figure~\ref{fig outerplanarmax}) as a spanning subgraph. We extend this idea to characterize every outerplanar graph that can be oriented as an oclique in the following result. 

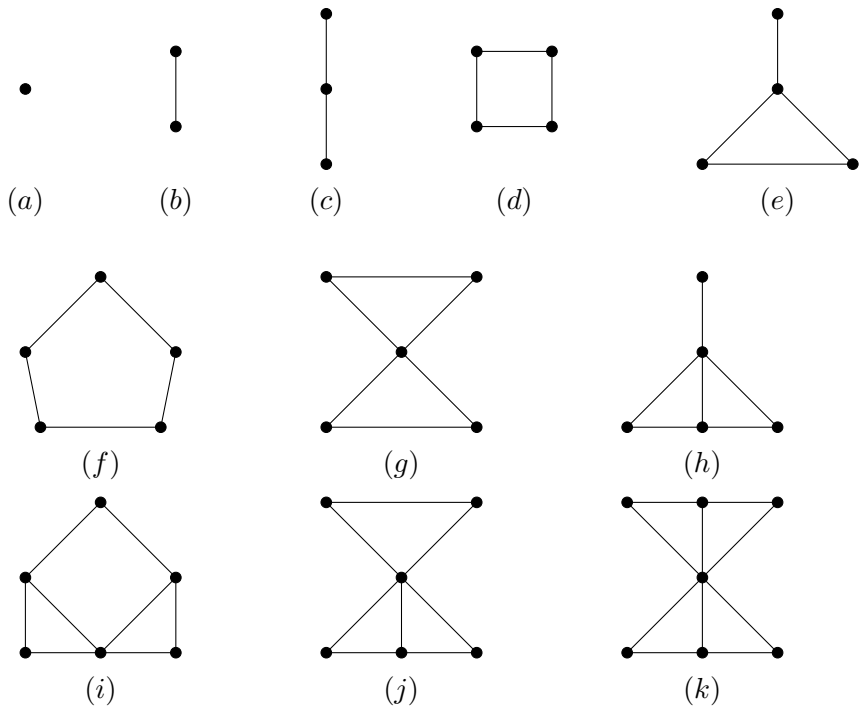
\begin{figure}

\centering
\begin{tikzpicture}

\filldraw [black] (0,6) circle (2pt) {node[left]{}};

\node at (0,4.5) {$(a)$};

\filldraw [black] (2,6.5) circle (2pt) {node[above]{}};
\filldraw [black] (2,5.5) circle (2pt) {node[below left]{}};

\draw[-] (2,6.5) -- (2,5.5);

\node at (2,4.5) {$(b)$};

\filldraw [black] (4,7) circle (2pt) {node[above]{}};
\filldraw [black] (4,6) circle (2pt) {node[below left]{}};
\filldraw [black] (4,5) circle (2pt) {node[below left]{}};

\draw[-] (4,7) -- (4,5);

\node at (4,4.5) {$(c)$};

\filldraw [black] (6,6.5) circle (2pt) {node[above]{}};
\filldraw [black] (6,5.5) circle (2pt) {node[below left]{}};
\filldraw [black] (7,6.5) circle (2pt) {node[below left]{}};
\filldraw [black] (7,5.5) circle (2pt) {node[below left]{}};

\draw[-] (6,6.5) -- (6,5.5);
\draw[-] (7,6.5) -- (7,5.5);
\draw[-] (6,6.5) -- (7,6.5);
\draw[-] (7,5.5) -- (6,5.5);

\node at (6.5,4.5) {$(d)$};

\filldraw [black] (9,5) circle (2pt) {node[above]{}};
\filldraw [black] (11,5) circle (2pt) {node[below left]{}};
\filldraw [black] (10,6) circle (2pt) {node[below left]{}};
\filldraw [black] (10,7) circle (2pt) {node[below left]{}};

\draw[-] (9,5) -- (11,5);
\draw[-] (9,5) -- (10,6);
\draw[-] (11,5) -- (10,6);
\draw[-] (10,6) -- (10,7);

\node at (10,4.5) {$(e)$};

\filldraw [black] (1,3.5) circle (2pt) {node[above]{}};
\filldraw [black] (0,2.5) circle (2pt) {node[above]{}};
\filldraw [black] (2,2.5) circle (2pt) {node[above]{}};
\filldraw [black] (0.2,1.5) circle (2pt) {node[above]{}};
\filldraw [black] (1.8,1.5) circle (2pt) {node[above]{}};

\draw[-] (1,3.5) -- (0,2.5) -- (0.2,1.5) -- (1.8,1.5) -- (2,2.5) -- (1,3.5);

\node at (1,1) {$(f)$};

\filldraw [black] (4,3.5) circle (2pt) {node[above]{}};
\filldraw [black] (6,3.5) circle (2pt) {node[above]{}};
\filldraw [black] (5,2.5) circle (2pt) {node[above]{}};
\filldraw [black] (4,1.5) circle (2pt) {node[above]{}};
\filldraw [black] (6,1.5) circle (2pt) {node[above]{}};

\draw[-] (4,3.5) -- (6,3.5);
\draw[-] (4,1.5) -- (6,1.5);

\draw[-] (5,2.5) -- (6,3.5);
\draw[-] (5,2.5) -- (4,3.5);
\draw[-] (5,2.5) -- (6,1.5);
\draw[-] (5,2.5) -- (4,1.5);

\node at (5,1) {$(g)$};

\filldraw [black] (9,3.5) circle (2pt) {node[above]{}};
\filldraw [black] (9,2.5) circle (2pt) {node[above]{}};
\filldraw [black] (10,1.5) circle (2pt) {node[above]{}};
\filldraw [black] (9,1.5) circle (2pt) {node[above]{}};
\filldraw [black] (8,1.5) circle (2pt) {node[above]{}};

\draw[-] (10,1.5) -- (8,1.5);

\draw[-] (9,2.5) -- (9,3.5);
\draw[-] (9,2.5) -- (8,1.5);
\draw[-] (9,2.5) -- (9,1.5);
\draw[-] (9,2.5) -- (10,1.5);

\node at (9,1) {$(h)$};

\filldraw [black] (1,.5) circle (2pt) {node[above]{}};
\filldraw [black] (0,-.5) circle (2pt) {node[above]{}};
\filldraw [black] (2,-.5) circle (2pt) {node[above]{}};
\filldraw [black] (1,-1.5) circle (2pt) {node[above]{}};
\filldraw [black] (0,-1.5) circle (2pt) {node[above]{}};
\filldraw [black] (2,-1.5) circle (2pt) {node[above]{}};

\draw[-] (1,.5) -- (0,-.5) -- (1,-1.5) -- (2,-.5) -- (1,.5);
\draw[-] (0,-.5) -- (0,-1.5) -- (2,-1.5) -- (2,-.5);

\node at (1,-2) {$(i)$};

\filldraw [black] (4,.5) circle (2pt) {node[above]{}};
\filldraw [black] (6,.5) circle (2pt) {node[above]{}};
\filldraw [black] (5,-.5) circle (2pt) {node[above]{}};
\filldraw [black] (4,-1.5) circle (2pt) {node[above]{}};
\filldraw [black] (5,-1.5) circle (2pt) {node[above]{}};
\filldraw [black] (6,-1.5) circle (2pt) {node[above]{}};

\draw[-] (4,.5) -- (6,.5);
\draw[-] (4,-1.5) -- (6,-1.5);

\draw[-] (5,-.5) -- (6,.5);
\draw[-] (5,-.5) -- (4,.5);
\draw[-] (5,-.5) -- (6,-1.5);
\draw[-] (5,-.5) -- (5,-1.5);
\draw[-] (5,-.5) -- (4,-1.5);

\node at (5,-2) {$(j)$};

\filldraw [black] (8,.5) circle (2pt) {node[above]{}};
\filldraw [black] (9,.5) circle (2pt) {node[above]{}};
\filldraw [black] (10,.5) circle (2pt) {node[above]{}};
\filldraw [black] (9,-.5) circle (2pt) {node[above]{}};
\filldraw [black] (10,-1.5) circle (2pt) {node[above]{}};
\filldraw [black] (9,-1.5) circle (2pt) {node[above]{}};
\filldraw [black] (8,-1.5) circle (2pt) {node[above]{}};

\draw[-] (10,.5) -- (8,.5);
\draw[-] (10,-1.5) -- (8,-1.5);

\draw[-] (9,-.5) -- (8,.5);
\draw[-] (9,-.5) -- (9,.5);
\draw[-] (9,-.5) -- (10,.5);
\draw[-] (9,-.5) -- (8,-1.5);
\draw[-] (9,-.5) -- (9,-1.5);
\draw[-] (9,-.5) -- (10,-1.5);

\node at (9,-2) {$(k)$};

\end{tikzpicture}

\caption{List of edge-minimal oclique spanning subgraphs of all outerplanar ocliques}\label{fig outerplanarocliques}
\end{figure}

\begin{theorem}~\label{outerall}
An undirected outerplanar graph can be oriented as an oclique if and only if it contains one of the 
graphs \ESOK{depicted in} Figure~\ref{fig outerplanarocliques} as a spanning subgraph.
\end{theorem}

We also prove a result similar to Theorem~\ref{gary-outer} for planar graphs which implies 
 Theorem~\ref{orientedplanarabsolute}(a) (that is, the \ESOK{absolute oriented} clique number of the family of planar graphs is 15).

\begin{theorem}\label{unique}
A planar oclique  has order at most 15 and every planar oclique of order 15 contains the planar oclique $\overrightarrow{P}$ depicted in Figure~\ref{fig planarmax} as a \ESOK{spanning} subgraph. 
\end{theorem}

\begin{figure} [t]

\centering

\centering
\begin{tikzpicture}

\filldraw [black] (0,0) circle (2pt) {node[below]{}};
\filldraw [black] (2,0) circle (2pt) {node[below]{}};
\filldraw [black] (4,0) circle (2pt) {node[below]{}};

\filldraw [black] (2,2) circle (2pt) {node[left]{}};

\filldraw [black] (0,4) circle (2pt) {node[above]{}};
\filldraw [black] (2,4) circle (2pt) {node[above]{}};
\filldraw [black] (4,4) circle (2pt) {node[above]{}};


\filldraw [black] (8,0) circle (2pt) {node[below]{}};
\filldraw [black] (10,0) circle (2pt) {node[below]{}};
\filldraw [black] (12,0) circle (2pt) {node[below]{}};

\filldraw [black] (10,2) circle (2pt) {node[right]{}};

\filldraw [black] (8,4) circle (2pt) {node[above]{}};
\filldraw [black] (10,4) circle (2pt) {node[above]{}};
\filldraw [black] (12,4) circle (2pt) {node[above]{}};

\filldraw [black] (6,2) circle (2pt) {node[above]{}};


\draw[->] (0,4) -- (1,3);
\draw[-] (1,3) -- (2,2);

\draw[->] (2,4) -- (2,3);
\draw[-] (2,3) -- (2,2);

\draw[->] (4,4) -- (3,3);
\draw[-] (3,3) -- (2,2);

\draw[->] (2,2) -- (1,1);
\draw[-] (1,1) -- (0,0);

\draw[->] (2,2) -- (2,1);
\draw[-] (2,1) -- (2,0);

\draw[->] (2,2) -- (3,1);
\draw[-] (3,1) -- (4,0);

\draw[->] (0,0) -- (1,0);
\draw[-] (1,0) -- (2,0);

\draw[->] (2,0) -- (3,0);
\draw[-] (3,0) -- (4,0);

\draw[->] (0,4) -- (1,4);
\draw[-] (1,4) -- (2,4);

\draw[->] (2,4) -- (3,4);
\draw[-] (3,4) -- (4,4);


\draw[->] (8,4) -- (9,3);
\draw[-] (9,3) -- (10,2);

\draw[->] (10,4) -- (10,3);
\draw[-] (10,3) -- (10,2);

\draw[->] (12,4) -- (11,3);
\draw[-] (11,3) -- (10,2);

\draw[->] (10,2) -- (9,1);
\draw[-] (9,1) -- (8,0);

\draw[->] (10,2) -- (10,1);
\draw[-] (10,1) -- (10,0);

\draw[->] (10,2) -- (11,1);
\draw[-] (11,1) -- (12,0);

\draw[->] (8,0) -- (9,0);
\draw[-] (9,0) -- (10,0);

\draw[->] (10,0) -- (11,0);
\draw[-] (11,0) -- (12,0);

\draw[->] (8,4) -- (9,4);
\draw[-] (9,4) -- (10,4);

\draw[->] (10,4) -- (11,4);
\draw[-] (11,4) -- (12,4);


\draw (2,0) .. controls (3,-2) and (4,-2) .. (6,2);

\draw[->] (3.2,-1.3) -- (3.2001,-1.3);

\draw (10,0) .. controls (9,-2) and (8,-2) .. (6,2);

\draw[->] (8.8,-1.3) -- (8.8001,-1.3);

\draw (2,4) .. controls (3,6) and (4,6) .. (6,2);

\draw[->] (3.2,5.3) -- (3.2001,5.3);

\draw (10,4) .. controls (9,6) and (8,6) .. (6,2);

\draw[->] (8.8,5.3) -- (8.8001,5.3);

\draw (0,0) .. controls (2,-4) and (4,-4) .. (6,2);

\draw[->] (2.7,-2.8) -- (2.7001,-2.8);

\draw (12,0) .. controls (10,-4) and (8,-4) .. (6,2);

\draw[->] (9.3,-2.8) -- (9.3001,-2.8);

\draw (0,4) .. controls (2,8) and (4,8) .. (6,2);

\draw[->] (2.7,6.8) -- (2.7001,6.8);

\draw (12,4) .. controls (10,8) and (8,8) .. (6,2);

\draw[->] (9.3,6.8) -- (9.3001,6.8);


\draw[->] (4,0) -- (5,1);
\draw[-] (5,1) -- (6,2);

\draw[->] (2,2) -- (4,2);
\draw[-] (4,2) -- (6,2);

\draw[->] (4,4) -- (5,3);
\draw[-] (5,3) -- (6,2);

\draw[->] (6,2) -- (7,1);
\draw[-] (7,1) -- (8,0);

\draw[->] (6,2) -- (8,2);
\draw[-] (8,2) -- (10,2);

\draw[->] (6,2) -- (7,3);
\draw[-] (7,3) -- (8,4);


\end{tikzpicture}

\caption{\ESOK{The planar} oclique $\vec{P}$ of order 15.}\label{fig planarmax}
\end{figure}
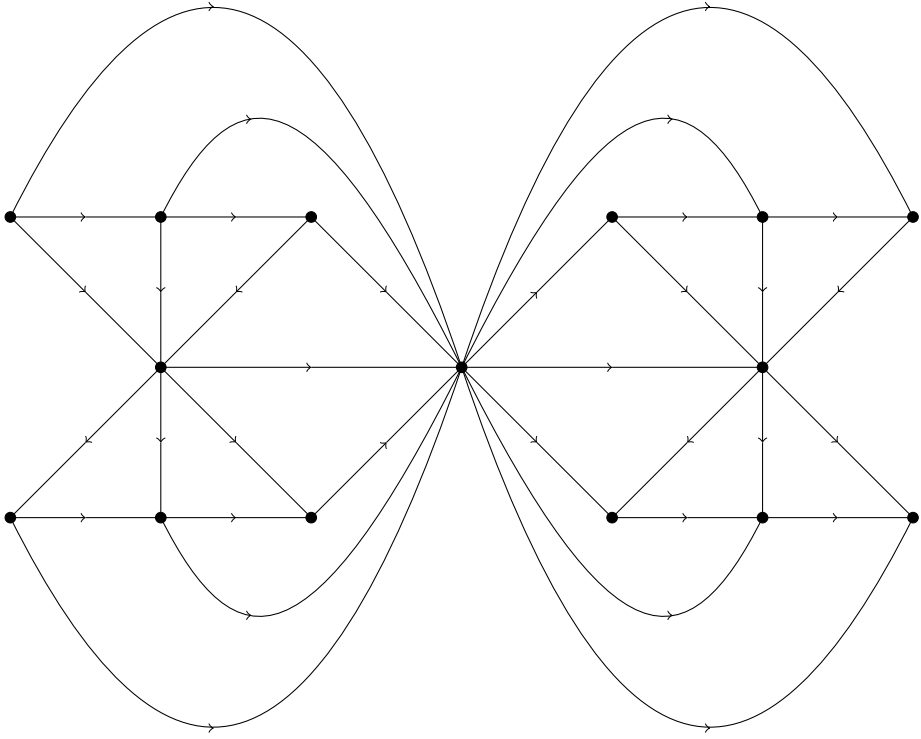

The question regarding the upper bound for \ESOK{the absolute oriented} clique number \ESOK{of the families} of  planar graphs with given girth (length of the smallest cycle in a graph)  is also of interest and was asked by  Klostermeyer and MacGillivray in~\cite{36}. 
We answer these questions and provide tight bounds.

Let $\mathcal{P}_k$ denote the family of planar graphs with girth at least $k$. \ESOK{We will prove the following.}

 \begin{theorem}\mbox{}\label{orientedplanarabsolute}

\begin{enumerate}[{\rm (a)}]
\item $\omega_{ao}( \mathcal{P}_3) = 15$.
\item $ \omega_{ao}( \mathcal{P}_4) = 6$.
\item $ \omega_{ao}( \mathcal{P}_5) = 5$.
\item $ \omega_{ao}( \mathcal{P}_k) = 3 $ for  $k \geq 6$.
\end{enumerate}
\end{theorem}

In Section~\ref{sec preliminary} we fix the notation to be used in this article and state some useful results. 
We also define the relative oriented clique number of an oriented graph, which will be used later in a proof. 
In Section~\ref{sec outerall}, \ref{sec unique} and \ref{sec orientedplanarabsolute}  we prove Theorem~\ref{outerall}, \ref{unique} and \ref{orientedplanarabsolute}, respectively. Finally, 
we mention in Section~\ref{sec conclusion} some future directions for research on this topic.  

\section{Preliminaries}\label{sec preliminary}

For an oriented graph $\vec{G}$, every parameter we introduce below is denoted using $\vec{G}$ as a subscript. In order to simplify \ESOK{notation}, this subscript will be dropped whenever there is no chance of confusion.

The set of all adjacent vertices of a vertex $v$ in an oriented graph $ \overrightarrow{G}$ is called its set of \textit{neighbors} and is denoted by $N_{\overrightarrow{G}}(v)$. 
If  \ESOK{$ \overrightarrow{uv}$ is an arc}, then $u$ is an \textit{in-neighbor} of $v$ and $v$ is an \textit{out-neighbor} of $u$. 
The set of all in-neighbors and the set of all out-neighbors of $v$ are denoted by $N_{\overrightarrow{G}}^-(v)$ 
 and $N_{\overrightarrow{G}}^+(v)$, respectively. 
The \textit{degree} of a vertex $v$ in an oriented graph $ \overrightarrow{G}$, denoted by 
${\rm deg}_{\overrightarrow{G}}(v)$, is the number of neighbors of $v$ in $ \overrightarrow{G}$. Naturally, the  \textit{in-degree} (resp. \textit{out-degree}) of a vertex $v$ in an oriented graph $ \overrightarrow{G}$, denoted by 
${\rm deg}^-_{\overrightarrow{G}}(v)$ (resp. ${\rm deg}^+_{\overrightarrow{G}}(v)$), is the number of in-neighbors (resp. out-neighbors) of $v$ in $ \overrightarrow{G}$.
The \textit{order} $ |\overrightarrow{G}|$ of an oriented graph $ \overrightarrow{G}$ is the cardinality of its
set of vertices $V(\overrightarrow{G})$.

\ESOK{We say that two} vertices $u$ and $v$ of an oriented graph \textit{agree} on a third vertex $w$ of that graph if $w  \in N^\alpha(u) \cap N^\alpha(v)$ for some $ \alpha \in \{+,-\}$
\ESOK{and that they \textit{disagree} on $w$ if $w  \in N^\alpha(u) \cap N^\beta(v)$ for $ \{ \alpha, \beta \} = \{+,-\}$.}

A \textit{directed path} of length $k$, or  a \textit{$k$-dipath}, from $v_0$ to $v_k$ \ESOK{is} an oriented graph with vertices $v_0, v_1, ..., v_k$ and arcs $ \overrightarrow{v_0v_1}, \overrightarrow{v_1v_2}, ..., \overrightarrow{v_{k-1}v_k}$ where $v_0$ and $v_k$ are  the \textit{terminal} vertices and $v_1, ..., v_{k-1}$ are \textit{internal} vertices. 
 A 2-dipath with arcs $ \overrightarrow{uv}$ and $ \overrightarrow{vw}$ is denoted by $ \overrightarrow{uvw}$. 
 More generally, 
 a 2-dipath with terminal vertices $u,w$ and internal vertex $v$ is denoted by $uvw$ (this denotes either 
 the 2-dipath $\overrightarrow{uvw}$ or the 2-dipath $\overrightarrow{wvu}$).
 A \textit{directed cycle} of length $k$, or a \ESOK{\textit{directed $k$-cycle}}, is an oriented graph with vertices 
$v_1, v_2, ..., v_k$ and arcs $ \overrightarrow{v_1v_2}, \overrightarrow{v_2v_3}, ..., \overrightarrow{v_{k-1}v_k}$ and $ \overrightarrow{v_kv_1}$. 

The \textit{directed distance} $ \overrightarrow{d}_{\overrightarrow{G}}(u,v)$ 
between two vertices 
$u$ and $v$ of an oriented graph $ \overrightarrow{G}$  is the smallest length of a  directed path \ESOK{ from
 $u$ to $v$ in} $ \overrightarrow{G}$.  We let $\overrightarrow{d}_{\overrightarrow{G}}(u,v)=\infty$ if no such directed path exists.
 The \textit{weak directed distance} $ \overline{d}_{\overrightarrow{G}}(u,v)$ 
between 
$u$ and $v$ is then given by $ \overline{d}_{\overrightarrow{G}}(u,v) = \min\{\overrightarrow{d}_{\overrightarrow{G}}(u,v),\overrightarrow{d}_{\overrightarrow{G}}(v,u)\}$.

Let now $G$ be an undirected graph. 
A \textit{path} of length $k$, or  a \textit{$k$-path}, from $v_0$ to $v_k$ \ESOK{is} a graph with vertices $v_0, v_1, ..., v_k$ and edges 
$ \overrightarrow{v_0v_1}, \overrightarrow{v_1v_2}, ..., \overrightarrow{v_{k-1}v_k}$.
The \textit{distance} $d_G(x,y)$ 
between two vertices $x$ and $y$  of 
$G$ is the smallest length of a path connecting $x$ and $y$.
The \textit{diameter} $diam(G)$ of a graph $G$ is the maximum distance between pairs of vertices of the graph. 

Triangle-free graphs with diameter 2 have been characterized by Plesn\'ik in~\cite{plesnik}.

\begin{theorem}[Plesn\'ik, 1975]\label{theorem plesnik}
The triangle-free graphs with diameter 2 are precisely the graphs listed in Figure~\ref{Plesnik}.
\end{theorem}

The graphs depicted in Figure~\ref{Plesnik} are the stars, the complete bipartite graphs  $K_{2,n}$
for some natural number $n$, and the graph obtained by adding copies of two non-adjacent vertices
of the 5-cycle.

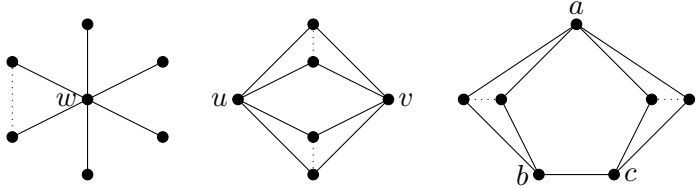
\begin{figure}[t]

\centering
\begin{tikzpicture}

\filldraw [black] (0,0) circle (2pt) {node[left]{}};
\filldraw [black] (0,1) circle (2pt) {node[left]{$w$}};
\filldraw [black] (0,2) circle (2pt) {node[right]{}};
\filldraw [black] (-1,.5) circle (2pt) {node[right]{}};
\filldraw [black] (-1,1.5) circle (2pt) {node[right]{}};
\filldraw [black] (1,.5) circle (2pt) {node[right]{}};
\filldraw [black] (1,1.5) circle (2pt) {node[right]{}};

\draw[dotted] (-1,1.5) -- (-1,.5);

\draw[-] (0,0) -- (0,1);
\draw[-] (0,2) -- (0,1);

\draw[-] (0,1) -- (-1,1.5);
\draw[-] (-1,.5) -- (0,1);

\draw[-] (0,1) -- (1,1.5);
\draw[-] (1,.5) -- (0,1);

\filldraw [black] (3,0) circle (2pt) {node[left]{}};
\filldraw [black] (3,.5) circle (2pt) {node[right]{}};
\filldraw [black] (3,1.5) circle (2pt) {node[right]{}};
\filldraw [black] (3,2) circle (2pt) {node[left]{}};

\filldraw [black] (2,1) circle (2pt) {node[left]{$u$}};
\filldraw [black] (4,1) circle (2pt) {node[right]{$v$}};

\draw[dotted] (3,0) -- (3,.5);
\draw[dotted] (3,2) -- (3,1.5);

\draw[-] (2,1) -- (3,2);

\draw[-] (2,1) -- (3,1.5);

\draw[-] (3,.5) -- (2,1);

\draw[-] (3,0) -- (2,1);

\draw[-] (3,2) -- (4,1);

\draw[-] (4,1) -- (3,1.5);

\draw[-] (4,1) -- (3,.5);

\draw[-] (3,0) -- (4,1);

 \filldraw [black] (6,0) circle (2pt) {node[left]{$b$}};
\filldraw [black] (7,0) circle (2pt) {node[right]{$c$}};

\filldraw [black] (5.5,1) circle (2pt) {node[left]{}};
\filldraw [black] (7.5,1) circle (2pt) {node[right]{}};

\filldraw [black] (5,1) circle (2pt) {node[left]{}};
\filldraw [black] (8,1) circle (2pt) {node[right]{}};

\filldraw [black] (6.5,2) circle (2pt) {node[above]{$a$}};

\draw[dotted] (5.5,1) -- (5,1);
\draw[dotted] (7.5,1) -- (8,1);

\draw[-] (6,0) -- (7,0);

\draw[-] (7,0) -- (7.5,1);

\draw[-] (7,0) -- (8,1);

\draw[-] (7.5,1) -- (6.5,2);

\draw[-] (8,1) -- (6.5,2);

\draw[-] (6.5,2) -- (5.5,1);

\draw[-] (6.5,2) -- (5,1);

\draw[-] (5.5,1) -- (6,0);

\draw[-] (5,1) -- (6,0);

\end{tikzpicture}

\caption{List of all triangle-free planar graphs with  diameter 2 (Plesn\'ik (1975)).}\label{Plesnik}

\end{figure}

A vertex subset $D$ is a \textit{dominating} set of a graph $G$ if every vertex of $G$ is either in $D$ or adjacent to a vertex of $D$.
The \textit{domination number} $ \gamma(G)$ of a graph $G$ is the minimum cardinality of a dominating set of $G$. 
  
We now define a new parameter for oriented graphs which will be used in our proof.
A \textit{\ESOK{relative oriented} clique}  of an oriented graph $\overrightarrow{G}$ is a set $R \subseteq V(\overrightarrow{G})$ of vertices 
such that any two vertices from $R$ are  at weak directed distance at most 2 in $ \overrightarrow{G}$. 
The 
\textit{\ESOK{relative oriented} clique number} $\omega_{ro}(\overrightarrow{G})$ of an oriented graph $ \overrightarrow{G}$ is the maximum order 
of a \ESOK{relative oriented} clique  of 
 $ \overrightarrow{G}$. 
 
 The  \ESOK{relative oriented} clique number $\omega_{ro}(G)$  of  a simple graph 
$ G$  is the maximum of the  \ESOK{relative oriented} clique numbers of all the oriented graphs with underlying graph  $G$. 
The \ESOK{relative oriented} clique number $\omega_{ro}(\mathcal{F})$ of  a family 
$ \mathcal{F}$ of graphs is the maximum of the  \ESOK{relative oriented} clique numbers of the graphs from the family $\mathcal{F}$.

 From the definitions,  we clearly have the following extension of Lemma~\ref{lem basic-in}.

\begin{lemma}\label{lem basic-in2}
For any oriented graph $\overrightarrow{G}$, 
$ \omega_{ao}(\overrightarrow{G})  \leq  \omega_{ro}(\overrightarrow{G})  \leq \chi_o(\overrightarrow{G})$.
\end{lemma}

 \ESOK{The} \ESOK{relative oriented} clique number of outerplanar graphs is at most 7 and this bound is tight.
 
 \begin{theorem}\label{theoremouterrelative}
 \ESOK{Let $\mathcal{O}$ be the family of outerplanar graphs. Then,} $\omega_{ro}(\mathcal{O}) = 7$.
 \end{theorem}

\begin{proof}
Note that the oriented outerplanar graph depicted in Figure~\ref{fig outerplanarmax} is an oclique. Hence, by Theorem~\ref{orientedouterplanar}
and Lemma~\ref{lem basic-in2} the result follows.
\end{proof}

\section{Proof of Theorem~\ref{outerall}}\label{sec outerall}

 It is easy to verify that  each graph 
 in Figure~\ref{fig outerplanarocliques} is outerplanar and can be oriented as an oclique. 
 If $G$ is an outerplanar graph \ESOK{which} contains any 
of the graphs from 
Figure~\ref{fig outerplanarocliques} as a spanning subgraph, then we can orient the edges of that spanning subgraph 
 to obtain an oclique and orient all the other edges arbitrarily. Note that with such an orientation $G$ is an outerplanar oclique.
 This proves the ``if'' part.
 
 \medskip

For the ``only if'' part, first note that a graph cannot be oriented as an oclique if it is disconnected. Now there are only two connected graphs with \ESOK{at most} two vertices, namely, the complete graphs $K_1$ (single vertex) and $K_2$ (an edge). Both of them are outerplanar and can be oriented as ocliques. Hence, any outerplanar graph on at most two vertices that can be oriented as an oclique must contain one of the graphs depicted in   Figure~\ref{fig outerplanarocliques}($a$) and Figure~\ref{fig outerplanarocliques}($b$) as a spanning subgraph.
  
\medskip  
  
If $G$ is connected  and has 3 vertices it must contain a 2-path as a spanning subgraph. We know that a 2-path is outerplanar and can be oriented as 
an oclique. Hence, any outerplanar graph on three vertices that can be oriented as an oclique must contain a 2-path (Figure~\ref{fig outerplanarocliques}($c$)) as a spanning subgraph.

\medskip

If $G$ has  at least 4 vertices and  is a  minimal (with respect to spanning subgraph inclusion) outerplanar graph that can be oriented as an oclique, 
then $ \Delta(G) \ge 2$ as every oriented tree is  3-colorable. Now we will do a case analysis to prove the remaining part. 

For the remainder of the proof assume that $V(G) = \{v_1, v_2, v_3, ..., v_{|G|}\}$ where $G$ is a  minimal (with respect to spanning subgraph inclusion) outerplanar graph of order at least 4 that can be oriented as an oclique.

\begin{enumerate}[(i)]

\item \textbf{For $|G| = 4$ and $\Delta(G) = 2$:} A triangle would force one vertex to have degree zero and hence $G$ can not be oriented as an oclique. So, $G$ must contain
a 4-cycle (Figure~\ref{fig outerplanarocliques}($d$)). 

\item  \textbf{For $|G| = 4$ and $\Delta(G) = 3$:} $G$ can not be $K_{1,3}$ since it is  3-colorable (as it is a tree). So we need to add at least one more edge to it. By adding one more edge to it, without loss of generality, we obtain the graph depicted in Figure~\ref{fig outerplanarocliques}($e$).

\item \textbf{For $|G| = 5$ and $\Delta(G) = 2$:}  $G$ must contain a $C_5$ (Figure~\ref{fig outerplanarocliques}($e$)) as any other connected graph on five vertices \ESOK{is either} a tree or has maximum degree \ESOK{greater} than 2.

\item \textbf{For $|G| = 5$ and $\Delta(G) = 3$:} Without loss of generality assume that ${\rm deg}(v_1)=3$ and $N(v_1) = \{v_2,v_3,v_4 \}$. 
If $|N(v_1) \cap N(v_5)| < 2$, then there cannot be a 2-dipath between $v_5$ and a vertex from $N(v_1) \setminus N(v_5)$. 
If $|N(v_1) \cap N(v_5)| > 2$, that is,  $N(v_1) =  N(v_5)$, then $G$ will contain $K_{2,3}$ as a subgraph which contradicts the \ESOK{outerplanarity} of $G$. 
Hence we must have $|N(v_1) \cap N(v_5)| = 2$. 

Without loss of generality assume that $N(v_5) = \{v_2,v_3\}$. 
Since $\Delta(G) =3$, $v_4$ must be adjacent to either $v_2$ or $v_3$ to allow a 2-dipath between $v_4$ and $v_5$.  But then $C_5$ is a spanning subgraph of $G$ contradicting the minimality of $G$.

\item \textbf{For $|G| = 5$ and $\Delta(G) = 4$:} First note that $G$ cannot have four edges as then it will be a tree and hence  3-colorable.  
Assume first that  $G$ has five edges. 
Without loss of generality assume that $N(v_1) = \{v_2, v_3, v_4, v_5\}$ and \ESOK{that} the fifth edge in $G$ \ESOK{is} $v_2v_3$. 
Now assume the arc  $\overrightarrow{v_2v_1}$ without loss of generality. 
This implies the arcs $\overrightarrow{v_1v_4}$, $\overrightarrow{v_1v_5}$  to have a 2-dipath between the pair of vertices $\{v_2, v_4\}$ and 
$\{v_2, v_5\}$ respectively. 
 But then it is not possible to connect $v_4$ and $v_5$ by a 2-dipath. Therefore, $G$ has at least six edges.
If $G$ has at least  six edges, then $G$ will contain  \ESOK{the graph depicted either in}  Figure~\ref{fig outerplanarocliques}($g$) or \ESOK{in}
 Figure~\ref{fig outerplanarocliques}($h$) as a spanning subgraph.

\item \textbf{For $|G| = 6$ and $\Delta(G) = 2$:} The only two connected \ESOK{graphs} on six vertices with $\Delta(G) = 2$ 
are the cycle $C_6$ on six vertices  and the path on six vertices. \ESOK{None} of them can be oriented as an oclique as 
the oriented chromatic number of trees and cycles \ESOK{is} bounded above by  3 and 5, respectively.

\item \textbf{For $|G| = 6$ and $\Delta(G) = 3$:}  Let $N(v_1) = \{v_2, v_3, v_4\}$. 
Since $G$ is 
an outerplanar graph we have $|N(v_1) \cap N(v_i)| \leq 2$ for $i \in \{5,6\}$. 
Now let $|N(v_1) \cap N(v_5)| = 0$. Then $v_2, v_3$ and $v_4$ must be connected by 2-dipaths to $v_5$ through $v_6$ contradicting 
$|N(v_1) \cap N(v_6)| \leq 2$. 
Hence $ |N(v_1) \cap N(v_5)| \in \{ 1, 2\}$.

First assume that  $|N(v_1) \cap N(v_5)| = 1$ and without loss of generality let $N(v_1) \cap N(v_5) = \{v_2\}$. 
Since $\Delta(G) = 3$, $v_2$ can be adjacent to at most one of the vertices from $\{v_3, v_4\}$. 
If $v_2$ is adjacent to exactly one vertex from $\{v_3,v_4\}$, say $v_3$ without loss of generality, then $v_4$ must be connected to $v_5$
by a 2-dipath through $v_6$. Now for having weak directed distance at most 2 between $v_3$ and $v_6$ we must either 
have the edge $v_3v_4$
or have the edge $v_3v_6$ 
(it is not possible to have the edge $v_2v_6$ as $\Delta(G) = 3$) 
creating a $K_4$-minor or a $K_{2,3}$ respectively, contradicting the outerplanarity of $G$. Hence,
$v_2$ is non-adjacent to both $v_3$ and $v_4$. In that case, $v_5$ must be connected to $v_3$ and $v_4$ by 2-dipaths through $v_6$. This will create a $K_{2,3}$-minor in $G$ contradicting its outerplanarity.

Now assume that $|N(v_1) \cap N(v_5)| = 2$ and let $N(v_1) \cap N(v_5) = \{v_2, v_3\}$. 
Since $G$ is outerplanar, $v_4$ cannot be connected to $v_5$ by a 2-dipath through $v_6$ as that will create a $K_{2,3}$-minor. 
So $v_4$ must be connected to $v_5$ by a  2-dipath through either $v_2$ or $v_3$. 
Note that $v_4$ cannot be adjacent to both $v_2$ and $v_3$ as it will create a $K_4$-minor contradicting the outerplanarity of $G$. 
Without loss of generality assume that $v_4$ is connected to $v_5$ by a 2-dipath through $v_3$. 
Now $v_6$ cannot be adjacent to $v_3$ as $\Delta(G) = 3$. 
To have weak directed distance from $v_6$ to $v_2$ and $v_4$ at most 2 we must have the edges $v_4v_6$ and $v_5v_6$. This will create 
a $K_{2,3}$-minor contradicting the outerplanarity of $G$. Hence we are done with this case as well.

\item \textbf{For $|G| = 6$ and $\Delta(G) = 4$:} 
Let $N(v_1) = \{v_2, v_3, v_4, v_5\}$. Then we have $|N(v_1) \cap N(v_6)| \le 2$ since $G$ is outerplanar. 
Let $N(v_1) \cap N(v_6) = N(v_6) = \{v_2\}$. 
Then $v_2$ can be a \ESOK{neighbor} of at most two vertices from $\{v_3,v_4,v_5\}$ in order to preserve outerplanarity of $G$ and hence the remaining vertex will not have any 2-dipath connecting it to $v_6$. 
So we must have $|N(v_6)| = 2$. 
Without loss of generality assume that $N(v_6) = \{v_3, v_4\}$. 
Now $v_6$ must be connected by 2-dipaths to $v_2$ and $v_5$. 
\ESOK{These} two 2-dipaths must go through $v_3$ or $v_4$. Both \ESOK{2-dipaths} cannot go through the same vertex $v_3$ (or $v_4$) as it 
will create 
a $K_{2,3}$-minor contradicting the outerplanarity of $G$. Therefore, one of the two 2-dipaths must go through  $v_3$ while the other must go through $v_4$. This will force the graph depicted in Figure~\ref{fig outerplanarocliques}($i$) to be a spanning subgraph of $G$ contradicting its minimality.

\item \textbf{For $|G| = 6$ and $\Delta(G) = 5$:} Assume that $N(v_1) = \{v_2,v_3, v_4, v_5, v_6\}$ and let the vertices $v_2,v_3, v_4, v_5$  and $v_6$
be arranged in \ESOK{clockwise order} around $v_1$ in a fixed planar embedding of $G$. 
\ESOK{Since $G$ is  outerplanar}, the \ESOK{induced subgraph} $G[v_2,v_3, v_4, v_5, v_6]$ \ESOK{cannot} have a cycle as it 
will create 
a \ESOK{$K_4$-minor,} contradicting the outerplanarity of $G$.
 
 Now without loss of generality assume that $|N^+(v_1)| > |N^-(v_1)|$. Note that if $|N^+(v_1)| \ge 4$ then we cannot have weak directed distance 
 \ESOK{at most}  2 between all the vertices of $N^+(v_1)$ keeping the graph outerplanar. Hence we must have $|N^+(v_1)| =3$ and $|N^-(v_1)| = 2$. 
 Now to have weak directed distance at most 2 between all the vertices of $N^+(v_1)$ and  between all the vertices of $N^-(v_1)$, keeping the graph 
 outerplanar, we must have 
  the graph depicted in Figure~\ref{fig outerplanarocliques}($j$) as a spanning subgraph of $G$.

\item \textbf{For $|G| = 7$:} It has been proved by Klostermeyer and MacGillivray in~\cite{36}  that $G$ must contain 
\ESOK{the graph depicted in} Figure~\ref{fig outerplanarocliques}($k$) as a spanning subgraph.
\end{enumerate} 

This concludes the proof. \hfill $\square$

 \section{Proof of Theorem~\ref{unique}}\label{sec unique}

Goddard and Henning~\cite{dom} proved that every planar graph of diameter 2 has domination number at most 2 except for
a particular  graph on nine vertices. 

Let $\overrightarrow{B}$ be a planar oclique  dominated by the vertex $v$. Sopena~\cite{ocoloring} showed that any oriented outerplanar graph has an oriented 
$7$-coloring (see Theorem~\ref{orientedouterplanar}). Hence let $c$ be an oriented $7$-coloring of the oriented outerplanar graph obtained from $\overrightarrow{B}$ by deleting the vertex $v$. For $u \in N^\alpha(v)$ let us assign the color $(c(u), \alpha)$ to  $u$ for $\alpha \in \{+,- \}$ and the color $0$ to $v$. It is easy to check that this is an oriented 
$15$-coloring of $\overrightarrow{B}$. Hence any planar oclique dominated by one vertex has order at most 15.

\begin{lemma}
Let $ \overrightarrow{H}$ be a planar oclique of order 15 dominated by one vertex.  Then $ \overrightarrow{H}$ contains the planar oclique  depicted in Figure~\ref{fig planarmax} as a \ESOK{spanning} subgraph. 
\end{lemma}

\begin{proof}
Suppose $ \overrightarrow{H}$ is a triangulated planar oclique of order 15 dominated by one vertex $v$.
Note that  $N^\alpha(v)$  is a \ESOK{relative oriented} clique in $\overrightarrow{H}[N(v)]$ (that is, the oriented subgraph of $\overrightarrow{H}$ induced by the neighbors of $v$, which is actually the oriented graph obtained by deleting the vertex $v$ from $\overrightarrow{H}$) for any $\alpha \in \{+,-\}$. Also note that $\overrightarrow{H}[N(v)]$
is an outerplanar graph. Hence, by Theorem~\ref{theoremouterrelative}, we have 
$|N^\alpha(v)| \leq 7$ for any $\alpha \in \{+,-\}$. But we also 
have
\begin{align}\nonumber
|N^+(v)| + |N^-(v)| = 14.\nonumber
\end{align} 
Hence we get
\begin{align}\nonumber
|N^+(v)| =|N^-(v)| =7.\nonumber
\end{align}

Now assume $N(v) = \{x_1,x_2,..., x_{14} \}$. 
Moreover, fix a planar embedding of 
$\overrightarrow{H}$ and, without  loss of generality, assume that the vertices $x_1, x_2, ..., x_{14}$ are arranged in a clockwise order around $v$. 
The triangulation of $\overrightarrow{H}$ forces the edges $x_1x_2$, $x_2x_3$, $\dots$, $x_{13}x_{14}$ and $x_{14}x_1$.
We know from the above discussion that there should be two disjoint \ESOK{relative oriented} cliques
$N^+(v)$ and $N^-(v)$, each of order 7, in the outerplanar graph $\overrightarrow{H}[N(v)]$. We already have the cycle $x_1x_2\dots x_{14}$ forced in the outerplanar graph $\overrightarrow{H}[N(v)]$.  We
will now prove  some more structural properties of $\overrightarrow{H}[N(v)]$.

As $\overrightarrow{H}[N(v)]$ is an outerplanar graph, it must have at least two vertices of degree at most 2.
As every vertex of the graph is part of a cycle, there is no vertex of degree at most 1. Hence there are at least two vertices of degree exactly 2 in $\overrightarrow{H}[N(v)]$.

Without loss of generality, assume that ${\rm deg}_{\overrightarrow{H}[N(v)]}(x_2) = 2$ 
and $x_2 \in N^\alpha(v)$ for some fixed $\alpha \in \{+,-\}$ and let $\{\alpha, \bar{\alpha}\} = \{+,-\}$.
Since $\overrightarrow{H}$ is triangulated, we must have the edge $x_1x_3$.
The vertices of $N^\alpha(v) \setminus \{x_1,x_2,x_3 \}$ must then be connected to $x_2$
by 2-dipaths with internal vertex either $x_1$ or $x_3$. 

Let four vertices of $N^\alpha(v) \setminus \{x_1,x_2,x_3 \}$  be connected to 
$x_2$ by 2-dipaths with internal vertex $x_1$. Then there will be two vertices, among the
above 
mentioned four vertices, at weak directed  distance at most 3 which is a contradiction. 
So at most three vertices of $N^\alpha(v) \setminus \{x_1,x_2,x_3 \}$ can be connected to 
$x_2$ by 2-dipaths with internal vertex $x_1$.
Similarly we can show that at most three vertices of $N^\alpha(v) \setminus \{x_1,x_2,x_3 \}$ 
can be connected to 
$x_2$ by 2-dipaths with internal vertex $x_3$.

Now suppose there are at least two vertices $x_i,x_j \in N^\alpha(v) \setminus \{x_1,x_2,x_3 \}$ 
that are connected to $x_2$ by 2-dipaths with internal vertex $x_1$ and there are at least
two vertices  $x_k, x_l \in N^\alpha(v) \setminus \{x_1,x_2,x_3 \}$ 
that are connected to $x_2$ by 2-dipaths with internal vertex $x_3$. 

Notice that, as the graph $\overrightarrow{H}$ is planar, with the given planar embedding of $\overrightarrow{H}$
we must have $i,j > k,l$. Now, without loss of generality, we can assume that $i > j$ and $k > l$.
But it will be impossible to have weak directed distance at most 2 between $x_i$ and $x_l$ keeping the graph $\overrightarrow{H}$ planar. 
So, at least one of the vertices $x_1$ or $x_3$ must be the internal vertex of at most one 2-dipath connecting a vertex of 
$N^\alpha(v) \setminus \{x_1,x_2,x_3 \}$ to $x_2$. 

If at least one of the vertices $x_1$ or $x_3$ belongs to $N^{\overline{\alpha}}(v)$, then 
we have
\begin{align}\nonumber
|N^\alpha(v) \setminus \{x_1,x_2,x_3 \}| \geq 5.\nonumber
\end{align}
But then, by the above discussion,
we will have a contradiction (there will be at least two vertices of 
$N^\alpha(v) \setminus \{x_1,x_2,x_3 \}$ connected by 2-dipaths with internal vertex $x_1$ and at least two vertices of 
$N^\alpha(v) \setminus \{x_1,x_2,x_3 \}$ connected by 2-dipaths with internal vertex $x_3$).

Hence we must have $x_1,x_3 \in N^\alpha(v)$. Without loss of generality, we have three vertices   
$x_i,x_j,x_k \in N^\alpha(v) \setminus \{x_1,x_2,x_3 \}$ connected by 2-dipaths with internal vertex $x_1$ and one vertex $x_l \in N^\alpha(v) \setminus \{x_1,x_2,x_3 \}$ connected by a 2-dipath with internal vertex $x_3$. Without loss of generality, we can assume $i > j > k > l$.

Now, to have $\overline{d}(x_i,x_s) \leq 2$ for $s \in \{2,3,l \}$, the vertices $x_2,x_3,x_l$ must disagree with $x_i$ on $x_1$. 
Also, to have $\overline{d}(x_i,x_k)\le 2$ and $\overline{d}(x_2,x_l) \leq 2$, we must have the 2-dipaths $x_ix_jx_k$ and $x_2x_3x_l$. But then  the 
induced oriented graph $\overrightarrow{H}[N^\alpha(v)]$ contains the oriented 
graph induced by $N^\alpha(a_0)$ of the planar oclique depicted in Figure~\ref{fig planarmax}.

Further notice that  no  vertex of $N^\alpha(v)$, other than $x_2$, has degree 2 in 
$\overrightarrow{H}[N(v)]$. Hence we can infer that a vertex of  $N^{\bar{\alpha}}(v)$ has  degree 2 in 
$\overrightarrow{H}[N(v)]$. That will imply that the induced oriented graph $\overrightarrow{H}[N^{\overline{\alpha}}(v)]$ contains the oriented 
graph induced by 
$N^{\overline{\alpha}}(a_0)$ of the planar oclique depicted in Figure~\ref{fig planarmax}.

Hence the planar oclique depicted in Figure~\ref{fig planarmax} is a subgraph of $\overrightarrow{H}$.
It is easy to check that, regardless of the choice of $\overrightarrow{H}$ (it is a triangulation of 
the planar oclique depicted in Figure~\ref{fig planarmax}), if we delete one arc of the oriented subgraph, isomorphic to the planar oclique depicted in Figure~\ref{fig planarmax}, of $\overrightarrow{H}$, 
 the 
oriented graph $\overrightarrow{H}$ does no longer remain an oclique. 
\end{proof}

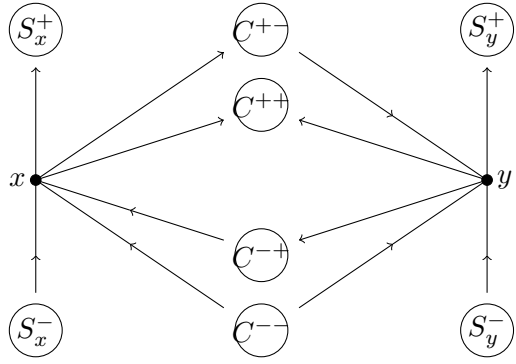
\begin{figure}

\centering
\begin{tikzpicture}

\draw [black] (3,0) circle (10pt) {node[]{$C^{--}$}};
\draw [black] (3,1) circle (10pt) {node[]{$C^{-+}$}};
\draw [black] (3,3) circle (10pt) {node[]{$C^{++}$}};
\draw [black] (3,4) circle (10pt) {node[]{$C^{+-}$}};

\draw [black] (0,0) circle (10pt) {node[]{$S_x^{-}$}};
\draw [black] (0,4) circle (10pt) {node[]{$S_x^{+}$}};

\draw [black] (6,0) circle (10pt) {node[]{$S_y^{-}$}};
\draw [black] (6,4) circle (10pt) {node[]{$S_y^{+}$}};

\filldraw [black] (0,2) circle (2pt) {node[left]{$x$}};
\filldraw [black] (6,2) circle (2pt) {node[right]{$y$}};

\draw[->] (0,.5) -- (0,1);
\draw[->] (0,1) -- (0,3.5);

\draw[->] (6,.5) -- (6,1);
\draw[->] (6,1) -- (6,3.5);

\draw[->] (0,2) -- (2.5,3.7);

\draw[->] (0,2) -- (2.5,2.8);

\draw[->] (2.5,1.2) -- (1.25,1.6);
\draw[-] (1.25,1.6) -- (0,2);

\draw[->] (2.5,.3) -- (1.25,1.15);
\draw[-] (1.25,1.15) -- (0,2);

\draw[->] (3.5,3.7) -- (4.75,2.85);
\draw[-] (4.75,2.85) -- (6,2);

\draw[->] (6,2) -- (3.5,2.8);

\draw[->] (6,2) -- (3.5,1.2);

\draw[->] (3.5,.3) -- (4.75,1.15);
\draw[-] (4.75,1.15) -- (6,2);

\end{tikzpicture}

\caption{Structure of $\vec{G}$ (not a planar embedding)}\label{not emb}

\end{figure}

Now, to prove Theorem~\ref{unique}, it will be enough to prove that every planar oclique of order 
at least 15 must have domination number 1. In other words, it will be enough to prove that 
any planar oclique with domination number 2 must have order at most 14. More precisely, we need to prove the following lemma.

\begin{lemma}\label{lem planar dom 2 implies order 14}
Let $ \overrightarrow{H}$ be a planar oclique with domination number 2.  Then $ |\overrightarrow{H} | \leq 14$. 
\end{lemma}

Let $ \overrightarrow{G}$ be a planar oclique with $|\overrightarrow{G}| > 14$. Assume that $ \overrightarrow{G}$ is triangulated and has domination number 2.

We define a partial order $\prec$ on the set of all dominating sets of order 2 of $ \overrightarrow{G}$ as follows: for any two
dominating sets $D = \{x,y\}$  and $D' = \{ x',y' \}$ of order 2 of $ \overrightarrow{G}$,  
$D' \prec D$ if and only if $ |N(x') \cap N(y')| < |N(x) \cap N(y)| $.

Let $D = \{x,y\}$ be a maximal dominating set of order 2 of $ \overrightarrow{G}$ with respect to $\prec$. Also, for the remainder of this section, let $t,t^\prime,\alpha, \overline{\alpha}, \beta, \overline{\beta}$ be variables satisfying $\{t,t^\prime \} = \{x, y\}$ and $\{\alpha, \overline{\alpha} \} = \{\beta, \overline{\beta} \} = \{+,- \}$.

Let us fix the following notations (see Figure~\ref{not emb}): 
\begin{align}\nonumber
C = N(x) \cap N(y) \text{, } C^{\alpha \beta} = N^{\alpha}(x) \cap N^{\beta}(y) \text{, } C_t^{\alpha} = N^{\alpha}(t) \cap C, \\
S_t = N(t) \setminus C \text{, } S^{\alpha}_t = S_t \cap N^{\alpha}(t) \text{, } S = S_x \cup S_y. \nonumber
\end{align}

Hence we have
\begin{equation}\label{order}
15 \leq | \overrightarrow{G} | = | D | + | C | + | S |.
\end{equation}
 
\begin{figure}

\centering
\begin{tikzpicture}

\filldraw [black] (0,2) circle (2pt) {node[left]{$x$}};
\filldraw [black] (8,2) circle (2pt) {node[right]{$y$}};

\filldraw [black] (4,4) circle (2pt) {node[above]{$c_0$}};

\filldraw [black] (4,3.2) circle (2pt) {node[below]{$c_1$}};

\filldraw [black] (4,2.5) circle (2pt) {node[below]{$c_2$}};

\filldraw [black] (4,1.5) circle (2pt) {node[below]{$c_i$}};

\filldraw [black] (4,.5) circle (2pt) {node[below]{$c_{k-2}$}};

\filldraw [black] (4,-.3) circle (2pt) {node[below]{$c_{k-1}$}};


\draw[-] (0,2) -- (4,-.3);
\draw[-] (0,2) -- (4,3.2);
\draw[-] (0,2) -- (4,4);
\draw[-] (0,2) -- (4,1.5);
\draw[-] (0,2) -- (4,2.5);
\draw[-] (0,2) -- (4,.5);

\draw[-] (8,2) -- (4,-.3);
\draw[-] (8,2) -- (4,3.2);
\draw[-] (8,2) -- (4,4);
\draw[-] (8,2) -- (4,1.5);
\draw[-] (8,2) -- (4,2.5);
\draw[-] (8,2) -- (4,.5);

\draw[dotted] (4,2.5) -- (4,.5);

\node at (3,3.9){$R_0$};

\node at (5,3.2){$R_1$};

\node at (4.8,.5){$R_{k-1}$};

\end{tikzpicture}

\caption{A planar embedding of $und( \vec{H})$ }\label{undh}

\end{figure}
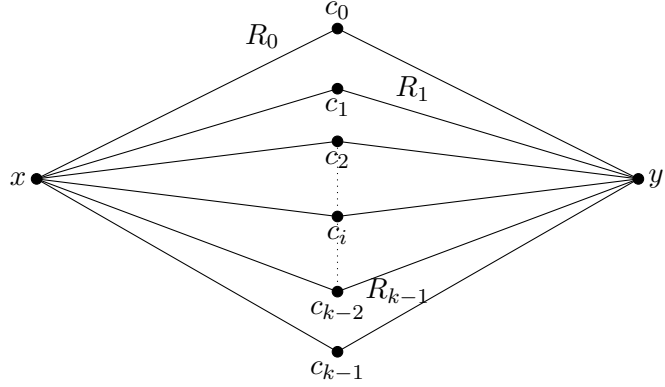

Let  $ \overrightarrow{H}$ be the oriented graph obtained from the induced subgraph $ \overrightarrow{G}[D \cup C]$ of $ \overrightarrow{G}$ by deleting all the arcs between the vertices of $D$ and  all the arcs 
 between the vertices of $C$.  Note that it is possible to extend  the planar embedding of $und(\overrightarrow{H})$ given in Figure~\ref{undh}  to a planar embedding of $und(\overrightarrow{G})$ for some particular ordering of the elements of, say  $C = \{c_0, c_1, \dots, c_{k-1} \} $.

Notice that $und( \overrightarrow{H})$ has $k$ faces, namely the unbounded face $F_0$ 
and the faces $F_{i}$ bounded by edges 
$xc_{i-1}$, $c_{i-1}y$, $yc_{i}$ and $c_{i}x$ for $i \in \{ 1, \dots, k-1 \}$. 
Geometrically, $und( \overrightarrow{H})$  divides the plane into $k$ connected components. The \textit{region}  $R_i$ of $ \overrightarrow{G}$  is the $i^{th}$ connected component (corresponding to the face $F_i$) of the plane.
The \textit{boundary points} of a region $R_i$ are $c_{i-1}$ and $c_i$ for $i \in \{ 1, \dots, k-1 \}$, and $c_0$ and $c_{k-1}$ for $i=0$. Two regions are \textit{adjacent} if  they have at least one common boundary point (hence, a region is adjacent to itself).

Now, for the different possible values of $| C |$, we want to show that  $und( \overrightarrow{H})$ cannot be extended to a planar oclique of order at least $15$. Note that, for extending $und( \overrightarrow{H})$ to $ \overrightarrow{G}$, we can add new vertices only from $S$. Any vertex $v \in S$ will be inside one of the regions $R_i$. If there is at least one vertex of $S$ in a region $R_i$, then $R_i$ is \textit{non-empty} and \textit{empty} otherwise. 
In fact,
when there is no chance of confusion,  $R_i$ might represent the set of vertices of $S$ contained in the region $R_i$. 

As any two distinct non-adjacent vertices of $ \overrightarrow{G}$ must be connected by a 2-dipath, we have the following three lemmas:

\begin{lemma}\label{adjacent}

\begin{enumerate}[{\rm (a)}]
\item If $(u,v)\in S_x \times S_y$ or $(u,v)\in S^\alpha_t \times S^\alpha_t$, 
 then $u$ and $v$ are in adjacent regions.
\item If $(u,c) \in S^\alpha_t \times C^\alpha_t$, 
 then $c$ is a boundary point of a region adjacent to the region containing $u$.
 \end{enumerate}
\end{lemma}

\begin{lemma}\label{connection}
Let $R$, $R^1$ and $R^2$ be three distinct regions such that $R$ is adjacent to  $R^i$  
with common boundary point $c^i$ while the other boundary point of 
$R^i$ is $ \overline{c^i}$, for all $i \in \{1,2 \}$. 
If $v \in S^\alpha_t \cap R$ and 
$ u^i \in ((S^\alpha_t \cup S_{t'}) \cap R^i) \cup (\{\overline{c^i} \} \cap C^\alpha_t) $, then $v$ 
disagrees with $u^i$ on $c^i$, for all 
$i\in \{1,2 \}$. Moreover, if both $u^1$ and $u^2$ exist, then 
$|S^\alpha_t \cap R| \leq 1$. 
\end{lemma}

\begin{lemma}\label{3same}
For any arc $ \overrightarrow{uv}$ in $ \overrightarrow{G}$,  $|N^\alpha(u) \cap N^\beta(v)| \leq 3$. 
\end{lemma}


Now we ask the question ``How small $|C|$ can be?'' and try to prove possible lower bounds  of
$|C|$. The first result regarding the lower bound of $|C|$ is proved below. 

\begin{lemma}\label{g2}
$| C | \geq 2$.
\end{lemma}

\begin{proof}
We know that $x$ and $y$ are either connected by a 2-dipath or by an arc. If $x$ and $y$ are adjacent then, as $ \overrightarrow{G}$ is triangulated, we have  $ | C | \geq 2$. If $x$ and $y$ are non-adjacent, then  $| C | \geq 1$. Hence it is enough to show that we cannot have $ | C | =1$ while $x$ and $y$ are non-adjacent.

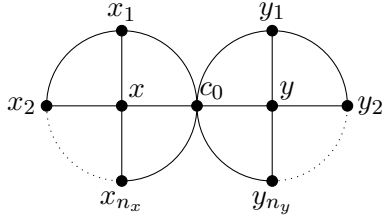
\begin{figure}

\centering
\begin{tikzpicture}

\filldraw [black] (1,1) circle (2pt) {};
\filldraw [black] (3,1) circle (2pt) {};

\filldraw [black] (2,1) circle (2pt) {};

\filldraw [black] (1,2) circle (2pt) {node[above]{$x_1$}};
\filldraw [black] (0,1) circle (2pt) {node[left]{$x_2$}};
\filldraw [black] (1,0) circle (2pt) {node[below]{$x_{n_x}$}};

\filldraw [black] (3,2) circle (2pt) {node[above]{$y_1$}};
\filldraw [black] (4,1) circle (2pt) {node[right]{$y_2$}};
\filldraw [black] (3,0) circle (2pt) {node[below]{$y_{n_y}$}};

\draw[-] (1,1) -- (2,1);
\draw[-] (1,1) -- (1,2);
\draw[-] (1,1) -- (0,1);
\draw[-] (1,1) -- (1,0);

\draw[-] (3,1) -- (2,1);
\draw[-] (3,1) -- (3,2);
\draw[-] (3,1) -- (4,1);
\draw[-] (3,1) -- (3,0);


\draw (1,0) arc (-90:180:1cm);
\draw[dotted] (0,1) arc (-180:-90:1cm);

\draw (4,1) arc (0:270:1cm);
\draw[dotted] (3,0) arc (-90:0:1cm);

\node at (1.2,1.2){$x$};
\node at (3.2,1.2){$y$};
\node at (2.2,1.2){$c_0$};

\end{tikzpicture}

\caption{For $|C|=1$ while $x$ and $y$ are non-adjacent}\label{c1}

\end{figure}

If $ | C | =1$ and $x$ and $y$ are non-adjacent, then the triangulation of $\overrightarrow{G}$ will force the  
configuration depicted in Figure~\ref{c1} as a subgraph of $und(\overrightarrow{G})$, where 
$C = \{ c_o \}$, $S_x = \{x_1, \dots, x_{n_x} \}$ and $S_y = \{y_1, \dots, y_{n_y} \}$. 
Without loss of generality, we may assume $| S_y | \geq | S_x | $. 
Then, by equation~(\ref{order}), we have
\begin{align}\nonumber
 n_y = | S_y | \geq \lceil (15-2-1) / 2 \rceil = 6. \nonumber
\end{align}
Clearly $n_x = | S_x |\geq 3$, as otherwise $ \{c_0,y\}$ would be a dominating set with at least two common neighbors 
$ \{y_1,y_{n_y} \}$, which contradicts the maximality of $D$. 

For $n_x=3$, we know that $c_0$ is not adjacent to $x_2$ as otherwise $ \{c_0,y\}$ would be a dominating set with at least two common neighbors $ \{y_1,y_{n_y} \}$,  contradicting the maximality of $D$. 
But then $x_2$ should be adjacent to $y_i$ for some $i \in \{1, \dots, n_y\}$, as otherwise $d(x_2,y) > 2$. 
Now the triangulation of $\overrightarrow{G}$ forces $x_2$ and $y_i$ to have at least two common neighbors. 
Also, $x_2$ cannot be adjacent to $y_j$ for any
$j \neq i$, as it would create a dominating set $ \{x_2,y \}$ with at least two common neighbors 
$ \{y_i,y_j \}$, contradicting the maximality of $D$. 
Hence, $x_2$ and $y_i$ are adjacent to both $x_1$ and $x_3$. 
  Note that $t_{\ell_t}$ and $t_{\ell_t +k}$ are adjacent if and only if $k=1$, as otherwise 
 $d(t_{\ell_t +1},t') > 2$ for $1 \leq \ell_t < \ell_t +k \leq n_t$.
 In this case, by equation~(\ref{order}), we have
\begin{align}\nonumber
 n_y = | S_y | \geq 15 - 2 - 1 - 3 =9. \nonumber
\end{align}

  Assume $i \geq 5$. 
  Hence, $c_0$ is adjacent to $y_j$ for all $j = 1, 2, 3$, as otherwise $d(y_j,x_3) > 2$. 
  This implies $d(y_2,x_2) > 2$, a contradiction. Similarly  $i < 5$ would also force  a contradiction.
  Hence $n_x \geq 4$.

 For $n_x = 4$, $c_0$ cannot be adjacent to both $x_3$ and $x_{n_x -2}=x_2$ as it would create a dominating set 
$ \{c_0,y \}$ with at least two common neighbors $ \{y_1,y_{n_y} \}$, contradicting the maximality of $D$. 
 For $n_x \geq 5$,  $c_0$ is adjacent to $x_3$ implies,  either for all $i \geq 3$ or for all $i \leq 3$, $x_i$ is adjacent to $c_0$, as otherwise $d(x_{i},y) > 2$. 
 Either of these cases would force  $c_0$ to become adjacent to $y_j$, as otherwise we would have either $d(x_{1},y_j) > 2$ or $d(x_{n_x},y_j) > 2$ for all 
 $j \in \{1, 2, \dots, n_y \} $. But then we would have a dominating set $ \{c_0,x \}$ with at least two common vertices,  contradicting the maximality of $D$. 
 Hence, for $n_x \geq 5$, $c_0$ is not adjacent to $x_3$. Similarly we can show, for $n_x \geq 5$, that $c_0$ is neither adjacent to  $x_3$ nor to $x_{n_x -2}$.
 
 So, for $n_x \geq 4$, we can assume without loss of generality that $c_0$ is not adjacent to $x_3$. We know that 
 $d(y_1,x_3) \leq 2$. We have already noted that 
 $t_{l_t}$ and $t_{l_t +k}$ are adjacent if and only if $k=1$ for any 
 $0 \leq l_t < l_t +k \leq n_t$. Hence, to have $d(y_1,x_3) \leq 2$, we must have one of the following edges: 
 $ y_1x_2$,   $ y_1x_3$,  $ y_1x_4$ or  $ y_2x_3$. 
 
 The first edge would imply the edges $ x_2y_j$ as otherwise $d(x_{1},y_j) > 2$ for all $j =3, 4, 5$. These three edges would then imply  $d(x_{4},y_3) > 2$. Hence we do not have 
 the edge $ y_1x_2$. 
 
 The other three edges, assuming we cannot have the edge $y_1x_2$,   would force the edges
 $x_2c_0$ and $x_1c_0$ for having $ d(x_2,y) \leq 2$ and $ d(x_1,y) \leq 2$. This  would imply $ d(x_1,y_4)>2$, a contradiction. Therefore, we cannot have the other three edges too.
 
 Hence we are done.
    \end{proof}

We now prove that, for $2 \leq |C| \leq 5$, at most one region of $ \overrightarrow{G}$ can be non-empty. Later, using this result, we will improve the lower bound of $|C|$.

\begin{lemma}\label{1region}
If $ 2 \leq | C | \leq 5$, then at most one region of $ \overrightarrow{G}$ is non-empty.
\end{lemma}

\begin{proof}
 (For pictorial help, refer to Figure~\ref{undh}.) 
 If $|C| = 2$ and $x$ and $y$ are adjacent, then the region that contains the edge $xy$ is empty, as otherwise the triangulation of $\overrightarrow{G}$ 
 would force $x$ and $y$ to have a common neighbor other than $c_0$ and $c_1$. 
 So, for the rest of the proof, we can assume $x$ and $y$ are non-adjacent if $|C| =2$.

\medskip

\noindent\textbf{\textit{Step 0.}} We first  show that it is not possible to have either $S_x=\emptyset$ or $S_y= \emptyset$ and have at least two non-empty regions. 
Without loss of generality, assume that $S_x = \emptyset$. Then $x$ and $y$ are non-adjacent, as otherwise $y$ would be a dominating vertex which is not possible. 

For $|C|=2$, if both $S_y \cap R_0$ and $S_y \cap R_1$ are non-empty, then the triangulation of $\overrightarrow{G}$  forces either two parallel edges $c_0c_1$ (one in each region) or a common neighbor of $x$ and $y$ other than $c_0,c_1$, a contradiction.

For $|C|=3,4$ and $5$,  the triangulation of $\overrightarrow{G}$ implies the edges $c_0c_1$, $\dots$, $c_{k-2}c_{k-1}$ and $c_{k-1}c_0$. 
Hence every $v\in S_y$ must be connected to $x$ by a 2-dipath through $c_i$ for some $i \in \{1, 2, \dots, k-1 \}$. Now assume 
$|S^\alpha_y| \geq |S_y^{\overline{\alpha}}|$ for some $\alpha \in \{+,- \}$. 
Then, by equation~(\ref{order}), we have
\begin{align}\nonumber
|S^\alpha_y| \geq \lceil (15-2-5)/2 \rceil =4. \nonumber
\end{align}
By Lemma~\ref{adjacent}, we know that the vertices of $S_y^\alpha$ will be contained in two adjacent regions for $|C|=4,5$. For $|C|= 3$, $S^\alpha_y \cap R_i \neq \emptyset $ for all $i \in \{0,1,3 \}$ implies $|S^\alpha_y| \leq 3$ by Lemma~\ref{connection}. Hence, without loss of generality, we may assume $S_y^\alpha \subseteq R_1 \cup R_2$. If both $S_y^\alpha \cap R_1$ and 
$S_y^\alpha \cap R_2$ are non-empty then, by Lemma~\ref{connection}, each vertex of $S_y^\alpha \cap R_1$ disagrees with each 
vertex of $S_y^\alpha \cap R_2$ on $c_1$. But then, $\{c_1,y\}$ becomes a dominating set with at least six common neighbors (namely $c_0$, $c_2$, and four vertices from $S_y^\alpha$), 
which contradicts the maximality of $D$. 

Hence, all the vertices of $S_y^\alpha$ must be contained in one region, say $R_1$. Each of them should then be connected to $x$ by a 2-dipath with internal vertex either $c_0$ or $c_1$. 
However, the vertices that are  
connected to $x$ by a 2-dipath with internal vertex $c_0$ should have weak directed distance 
at most 2 with the vertices
connected to $x$ by a 2-dipath with internal vertex  $c_1$. But it is not possible to connect them 
unless they are all adjacent to either $c_0$ or $c_1$, in which case it would contradict the maximality of $D$. 

 Hence both $S_x$ and $S_y$ are non-empty. 

\medskip

\noindent\textbf{\textit{Step 1.}} We now prove that  at most four sets  out of the $2k$ sets $S_t \cap R_i $ can be non-empty, 
for all $t \in \{x,y \}$ and  $i \in \{0,1, \dots, k-1 \}$. 
It is immediate for  $|C| = 2$. 
For $|C|=4$ and $5$, the statement follows from Lemma~\ref{adjacent}. For $|C|=3$, we consider the following two cases: 

\begin{enumerate}[{\rm (i)}]

\item Assume   $S_t \cap R_i \neq \emptyset$ for all $t \in \{x,y \}$ and for all $i \in \{0,1,2 \}$. Then, by Lemma~\ref{connection}, we have  $|S_t \cap R_i| \leq 1$  for all $t \in \{x,y \}$ and for all $i \in \{0,1,2 \}$. By equation~(\ref{order}), we have
\begin{align}\nonumber
15 \leq | \overrightarrow{G}| = 2 + 3 + 4 = 9. \nonumber
\end{align}
This is a contradiction. 

\item Assume that five out of the six sets $S_t \cap R_i $ are non-empty and that the other one is empty, where $t \in \{x,y \}$ and  $i \in \{0,1,2 \}$. Without loss of generality, 
we can assume $S_x \cap R_0 = \emptyset $.
By Lemma~\ref{connection}, we have $|S_t \cap R_i| \leq 1$ for all 
$(t,i) \in \{(x,1), (x,2), (y,0) \}$. In particular, $|S_x| \leq 2$.

Now, all vertices of $S_t \cap R_i$ are adjacent to $c_1$ for $i \in \{1,2\}$, for being at weak directed distance at most 2 from each other, by Lemma~\ref{connection}. That means every vertex of $S_x$ is adjacent to $c_1$. Hence, there can be at most 
three vertices in $(S_y \cap R_1) \cup (S_y \cap R_2)$ as otherwise the dominating set $\{c_1,y\}$ would contradict the maximality of $D$.   Hence, $|S_y | \leq 4$.

Therefore, by equation~(\ref{order}) we have
\begin{align}\nonumber
 15 \leq | \overrightarrow{G}| = 2+3+(2+4) =11.\nonumber
 \end{align}
 This is a contradiction. 
\end{enumerate}

Hence,  at most four sets  out of the $2k$ sets $S_t \cap R_i $ can be non-empty, 
where $t \in \{x,y \}$ and  $i \in \{0,1, \dots, k-1 \}$.

\medskip

\noindent\textbf{\textit{Step 2.}} Assume now that exactly four sets out of the  sets $S_t \cap R_i $ are non-empty, for all $t \in \{x,y \}$ and  $i \in \{0, \dots, k-1 \}$. Without loss of generality, 
we have the following three cases (by Lemma~\ref{adjacent}): 

\begin{enumerate}[{\rm (i)}]

\item Assume the four non-empty sets are $S_x \cap R_1, S_y \cap R_0, S_y \cap R_1$ and $S_y \cap R_2$ (only possible for $|C|\geq 3$). The triangulation of $\overrightarrow{G}$
then forces the edges $c_0c_{k-1}$ and $c_1c_2$.
 Lemma~\ref{connection} implies that $S_x \cap R_1 = \{x_1 \}$ and that the vertices of $S_y \cap R_0$ and the vertices of $S_y \cap R_2$ disagree with $x_1$ on $c_0$ and $c_1$, respectively.

For $|C|=3$,  if every vertex from $S_y \cap R_1$ is adjacent to either $c_0$ or $c_1$, then $\{c_0,c_1 \}$ will be a dominating set with  at least four common neighbors $\{x,y,x_1,c_2 \}$, 
contradicting the maximality of $D$. 
If not, then the triangulation of $\overrightarrow{G}$ will force $x_1$ to be adjacent to at least two vertices, from $S_y$, say $y_1$ and $y_2$. 
But then, $\{x_1,y \}$ would be  a dominating set with at least four common  neighbors $\{y_1,y_2,c_0,c_1 \}$, contradicting the maximality of $D$.

For $|C| = 4$ and $5$, Lemma~\ref{adjacent} implies that vertices of $S_y \cap R_0$ and vertices of $S_y \cap R_2$ disagree with each other on $y$. 
Now, by Lemma~\ref{connection},  any vertex of $S_y \cap R_1$ is adjacent  either to $c_0$ (if it agrees with the vertices of $S_y \cap R_0$ on $y$) or to $c_1$ (if it agrees with the vertices of $S_y \cap R_2$ on $y$). 
Also, the vertices of $S_y \cap R_0$ and $S_y \cap R_2$ are connected to $x_1$ by a 2-dipath through $c_0$ and $c_1$ respectively. 
Hence, by Lemma~\ref{3same}, we have $|S_y \cap R_0|, |S_y \cap R_2| \leq 3$.

 Now, by equation~(\ref{order}), we have
 \begin{align}\nonumber
 |S_y| \geq (15-2-5-1) = 7.\nonumber
 \end{align}
  Hence, without loss of generality, at least four vertices $y_1,y_2,y_3,y_4$ of $S_y$ are adjacent to $c_0$. But, in that case, $\{c_0,y \}$ is a dominating set with at 
  least five common neighbors $\{y_1,y_2,y_3,y_4,c_{k-1} \}$, contradicting the maximality of $D$ for $|C|=4$.

  For $|C| = 5$, each vertex of
 $S_y \cap R_1$ disagrees with $c_3$ on $y$ by Lemma~\ref{adjacent} and therefore, without loss of generality,  all of them are adjacent to  $c_0$. 
 Then, vertices of $S_y \cap R_1$ disagrees with 
vertices of  $S_y \cap R_2$ on $y$ as well. This implies vertices of 
$S_y \cap R_2$ agrees with $c_3$ on $y$ and must be connected to $c_3$ by 2-dipaths with internal vertex $c_2$.
 Now, by Lemma~\ref{connection},  $|S_y \cap R_2| \leq 1$. So, $|S_y| \geq  7$ implies $|S_y \cap (R_0 \cup R_1)| \geq 6$. But every vertex of $S_y \cap (R_0 \cup R_1)$ are adjacent to $c_0$. 
 In that case, $\{c_0,y \}$ is a dominating set with at least six common neighbors,  contradicting the maximality of $D$ for $|C|=5$.

\item Assume the four non-empty sets are $S_x \cap R_0, S_x \cap R_1, S_y \cap R_0$ and $S_y \cap R_1$.  
 For $|C| =2$, every vertex in $S$ is adjacent  either to $c_0$ or to $c_1$ (by Lemma~\ref{connection}). 
So, $\{ c_0,c_1 \}$ is a dominating set. Hence, no vertex $w\in S$ can be adjacent to both $c_0$ and $c_1$ since otherwise 
$\{c_0,c_1 \}$ would be a dominating set with at least three common neighbors $\{x,y,w \}$, contradicting the maximality of $D$. 
By equation~(\ref{order}), we have
\begin{align}\nonumber
| S | \geq 15-2-2 = 11.\nonumber
\end{align}
Hence, without loss of generality, we may assume  $| S_x \cap R_0 | \geq 3$. Suppose $\{x_1,x_2,x_3 \} \subseteq S_x \cap R_0$.
 In that case, all vertices of $S_x \cap R_0$  must be adjacent to $c_0$ (or to $c_1$), as otherwise it would force all  vertices  of 
 $S_y \cap R_1$ to be adjacent to both $c_0$ and $c_1$ (by Lemma~\ref{connection}). Without loss of generality, assume that all vertices of $S_x \cap R_0$ are adjacent to $c_0$. 
 Then, all vertices
$w\in S_y$ will be adjacent to $c_0$, as otherwise $d(w,x_i) >2$, for some $i \in \{1, 2, 3 \}$. But then
$\{c_0, x \}$ would be a dominating set with at least three common vertices $\{x_1, x_2, x_3 \}$, contradicting the maximality of $D$.

For $|C|= 3,4$, every vertex of $S$ will be adjacent to $c_0$ (by Lemma~\ref{connection}). By 
equation~(\ref{order}), we have
\begin{align}\nonumber
|S| \geq (15-2-4) = 9.\nonumber
\end{align}
 Hence, without loss of generality, $|S_x| \geq 5$. In that case, $\{c_o,x \}$ is a dominating set with at least five common neighbors $S_x \cup \{y \}$, 
 contradicting the maximality of $D$ for $|C|=3,4$.

For $|C|=5$, every vertex of $S_t \cap R_i $ disagrees with $c_{i+2}$ on $t$ and, therefore, 
$|S_t \cap R_i| \leq 3 $ for  $i \in \{0,1 \}$ 
by Lemma~\ref{adjacent}.  Assume $|S_x \cap R_0| = 3$ and $S_x \cap R_0 = \{x_1,x_2,x_3 \}$.
Moreover, assume without loss of generality that $c_2 \in N^\alpha(x)$. In that case, we must have 
$\{x_1,x_2,x_3 \} \subseteq N^{\overline{\alpha}}(x)$. 

Note that $x_1$, $x_2$ and $x_3$ must agree on $c_0$ in order to be at weak directed distance at most 2 with the vertices of $S_y \cap R_1$. 
Further, assume 
that $\{x_1,x_2,x_3 \} \subseteq N^\beta(c_0)$. 
But then, as all the three vertices 
$\{x_1,x_2,x_3 \}$ are adjacent to both $x$ and $c_0$, the only way each of them can be 
at weak directed distance 2 from $c_3$ is through a 2-dipath with internal vertex $x$. 
Hence, we have  $c_3 \in N^\alpha(x)$. 
This implies  $x_4 \in N^{\overline{\alpha}}(x)$ for any vertex $x_4 \in S_x \cap R_1$. 
But then, the vertices of  $S_x \cap R_1$ must disagree with vertices of $S_x \cap R_0$ on $c_0$,
making it impossible for the vertices of $S_y \cap R_0$ to be at weak directed distance at most 2 from 
$x_1,x_2,x_3$ and  from the
vertices of $S_x \cap R_1$. Therefore, we must have  $|S_x \cap R_0| \leq 2$.

Similarly, we can prove $|S_t \cap R_i| \leq 2$ for  $i \in \{0,1 \}$.

We now show that it is not possible to 
have $|S_t \cap R_i| = 2$ for all $(t,i) \in \{x,y\} \times \{0,1 \}$.
Suppose on the contrary that this is the case. 
Then, clearly, the vertices of $S_t \cap R_i$ disagree with $c_{i+2}$ and $c_{i+3}$ on $t$.
Hence, the vertices of $S_t \cap R_0$ agree with the vertices of $S_t \cap R_1$ on $t$.
Therefore, the vertices of $S_t \cap R_0$ must disagree with the 
vertices of $S_t \cap R_1$ on $c_0$.

Then it will not be possible to have both the vertices of $S_x \cap R_0$ at weak directed distance at most 2 from all the four vertices of $S_y$. 

Therefore, we have $|S| \leq 7$. Hence, by 
equation~(\ref{order}), we have
\begin{align}\nonumber
15 \leq | \overrightarrow{G}| \leq 2+5+7 =14.\nonumber
\end{align}
This is a contradiction and we are done.

\item Assume the four non-empty sets are $S_x \cap R_1, S_x \cap R_2, S_y \cap R_0$ and $S_y \cap R_1$ (only possible for $|C|= 3$). 
In that case, Lemma~\ref{connection} implies that every vertex of $(S_x\cap R_1) \cup (S_y\cap R_0)$ is adjacent to $c_0$ and  that
 every vertex of $(S_x\cap R_2) \cup (S_y\cap R_1)$ is adjacent to $c_1$.

  Moreover, the triangulation of $\overrightarrow{G}$ forces the edges $c_0c_2$ and $c_1c_2$. It also forces some vertex $v_1 \in S_y \cap R_1$ to be adjacent to $c_0$. 
  But this would create the dominating set $\{c_0,c_1 \}$ with at least four common neighbors $\{x,y,v_1,c_2 \}$ contradicting the maximality of $D$. 
  
\end{enumerate}

  Hence  at most three sets  out of the $2k$ sets $S_t \cap R_i $ can be non-empty, 
where $t \in \{x,y \}$ and  $i \in \{0,1, \dots, k-1 \}$. 

\medskip

\noindent\textbf{\textit{Step 3.}} Now assume that exactly three sets out of the  sets $S_t \cap R_i $ are non-empty, where $t \in \{x,y \}$ and  $i \in \{0, \dots, k-1 \}$. Without loss of generality we have the following two cases (by Lemma~\ref{adjacent}): 

\begin{enumerate}[{\rm (i)}]

\item Assume the three non-empty sets are $S_x \cap R_0, S_y \cap R_0$ and $S_y \cap R_1$.  The triangulation of $\overrightarrow{G}$  implies that the edge $ c_0c_1$ lies inside the region $R_1$.  

\medskip

For $|C|=2$, there exists 
$u \in S_y \cup R_1$ such that $u$ is adjacent to both $c_0$ and $c_1$, by the triangulation of $\overrightarrow{G}$. 
Now,  if $|S_y \cup R_1| \geq 2$, then some other vertex $v \in S_y \cup R_1$
must be adjacent to either $c_0$ or $c_1$. Without loss of generality, we may assume that $v$ is adjacent to $c_0$. Then, every vertex $w \in S_x \cap R_0$ will be adjacent to $c_0$,
in order to have $d(v,w) \leq 2$. But in that case
$\{c_0, y \}$ would be a dominating set with at least three common neighbors $\{c_1, u, v \}$, contradicting the maximality of $D$.

So we must have $|S_y \cup R_1|= 1$. 
Assume that $S_y \cup R_1 = \{u \}$. 
Then, any vertex $w \in S_x \cap R_0$ is adjacent to either $c_0$ or $c_1$. If $|S_x | \geq 5$ then, without loss of generality, 
we can assume that at least three  vertices of $S_x$ are adjacent to $c_0$. Now, to have weak directed distance at most 2 from all those three vertices, every vertex of $S_y$ must be adjacent to $c_0$. 
This would create the dominating set $\{c_0, x \}$ with at least three common neighbors, contradicting the maximality of $D$.

 Also $|S_x | = 1$ clearly creates  the dominating set $\{c_0, y \}$ (as $x_1$ is adjacent to $c_0$ by the triangulation of $\overrightarrow{G}$) with at least three common neighbors (a vertex  from $S_y \cap R_0$ by the triangulation of $\overrightarrow{G}$, $u$ and $c_1$), contradicting the maximality of $D$.

For 
$2 \leq |S_x | \leq 4$, 
$c_0$ (or $c_1$) can be adjacent to at most two vertices of $S_y \cap R_0$  since otherwise there would be one vertex $v \in S_y \cap R_0$ which would force $c_0$ (or $c_1$) to be adjacent to all vertices of $w\in S_x$ (in order to satisfy $d(v,w) \leq 2$) and create a dominating set $\{c_0,y \}$ that contradicts the maximality of $D$.

 Also, not all vertices of $S_x$ can be adjacent to $c_0$ (or $c_1$), as otherwise $\{c_o,y \}$ (or $\{c_1,y \}$) would be a dominating set with at least three common neighbors ($u$, $c_1$ (or $c_0$) and a vertex  from $S_y \cap R_0$), contradicting the maximality of $D$.

 Note that, by equation~(\ref{order}), we have
 \begin{align}\nonumber
 |S_y \cap R_0| \geq 10 - S_x.\nonumber
 \end{align}
  Assume $S_x= \{x_1,\dots,x_n \}$, with the triangulation of $\overrightarrow{G}$ forcing the edges $c_0x_1$, $x_1x_2$, $\dots$, $x_{n-1}x_n$ and $x_nc_1$ for $n \in \{2,3,4 \}$.

 For $|S_x| = 2$, at most four vertices of $S_y \cap R_0$ can be adjacent to $c_0 $ or $c_1$. Hence, there will be at least four vertices of $S_y \cap R_0$ each connected to $x$ by a 2-dipath through $x_1$ or $x_2$. Without loss of generality , $x_1$ will be adjacent to at least 2 vertices of $S_y$, and hence $\{x_1,y \}$ will be a dominating set contradicting the maximality of $D$.

 For $|S_x| = 3$, without loss of generality, assume that $x_2$ is adjacent to $c_0$. To satisfy 
 $\overline{d}(x_1,v) \leq 2$ for all $v \in S_y \cap R_0$, at least four vertices 
 of $S_y$ will be  connected to $x_1$ by a 2-dipath through $x_2$ (as, according to previous discussions, at most two vertices of 
 $S_y$ can be adjacent to $c_0$). This would create the 
 dominating set $\{x_2,y \}$, contradicting the maximality of $D$.

 For $|S_x| = 4$ we have the edges $x_2c_0$ and $x_3c_1$, as otherwise at least three vertices of $S_x$ would be adjacent to either $c_0$ or  $c_1$, which is not possible (because it forces all vertices of $S_y$ to be adjacent to $c_0$ or $c_1$). Now, each vertex  $v \in S_y \cap R_0$ must be adjacent  either to $c_0$ or to $x_2$ (to satisfy $\overline{d}(v,x_1) \leq 2$) and also either to $c_1$ or to $x_3$ 
 (to satisfy $\overline{d}(v,x_4) \leq 2$), which is not possible due to the planarity of $\overrightarrow{G}$.

 For $|C| =3,4,5$, by Lemma~\ref{connection}, each vertex of $ S_x$ disagrees with each vertex of $S_y \cap R_1$ on $c_0$. 
 We also have the edge $x_1c_2$ for some $x_1 \in S_x$ by the triangulation of $\overrightarrow{G}$. By equation~(\ref{order}), we have
 \begin{align}\nonumber
 |S| \geq (15-2-|C|) = 13-|C|.\nonumber
 \end{align}
Hence, $|S_x|\leq 2$ for $|C|=3,4$, as otherwise every vertex $u \in S_y$ would be adjacent to 
$c_0$, creating a  dominating set 
$\{c_0,t \}$  with at least $(|C|+1)$ common neighbors $S_t \cup \{c_1 \}$ for some $t \in \{x,y \}$, 
contradicting the maximality of $D$. For $|C|=5$, since all the vertices in $S_x \cap R_0$ agree with each other
on $x$ (as they all must disagree  with $c_2$ on $x$) and on $c_0$ (as they all disagree with 
vertices of $S_y \cap R_1$ on $c_0$), 
 by Lemma~\ref{3same}, we have $|S_x \cap R_0| \leq 3$. But 
if $|S_x \cap R_0| = 3$ then every vertex of $S_y$ will be adjacent to 
$c_0$, creating a  dominating set 
$\{c_0,y \}$  with at least six common neighbors $S_y \cup \{c_1 \}$,
contradicting the maximality of $D$.

Hence  $|S_x|\leq 2$ for $|C|=3,4,5$.

\medskip

Now for $|C|=3$, we can assume that $x$ and $y$ are non-adjacent as otherwise $\{c_0,y \} $ would be a dominating set with at least four common neighbors 
($x$, $c_1$ and two other vertices each from the sets $S_y \cap R_0$, $S_y \cap R_1$ by triangulation) contradicting the maximality of $D$. Hence triangulation will imply the edge $c_1c_2$. Now for $|S_x|\leq 2$,  either  $\{c_0,c_2 \}$ is a dominating set  with at least four common neighbors $\{x,y,c_1,x_1 \}$  contradicting the maximality of $D$ or  $x_1$ is adjacent to at least two vertices  $y_1,y_2 \in S_y \cap R_0$ creating a dominating set $\{x_1,y \}$ 
(the other vertex in $S_x$  must be adjacent to $x_1$ by triangulation) 
with at least four common neighbors $\{y_1,y_2,c_0,c_2 \}$ contradicting the maximality of $D$.

For $|C| = 4$ we have $|S_y \cap R_1| \leq 2$ as otherwise we will have the dominating set $\{c_0,y \}$ with at least 
five common neighbors ($c_1$, vertices of $S_y \cap R_1$ and one vertex of $S_y \cap R_0$ by the triangulation of $\overrightarrow{G}$), contradicting the maximality of $D$. 
By equation~(\ref{order}), we have
\begin{align}\nonumber 
|S_y \cap R_0| &\geq (15 -|D| - |C|- |S_x| - |S_y \cap R_1|)\\ \nonumber
 &\geq (15-2-4-2-2) = 5.\nonumber
\end{align}
Now, at most two vertices of $S_y \cap R_0$ can be adjacent to $c_0$ as otherwise $\{c_0,y \}$ would be a dominating set with at least 
five common neighbors 
($c_1$, vertices of $S_y \cap R_0$ and one vertex of $S_y \cap R_1$ by the triangulation of $\overrightarrow{G}$),
 contradicting the maximality of $D$.

 Also, by the triangulation of $\overrightarrow{G}$, in $R_3$ we  have either the edge $xy$ or  the edge $c_2c_3$.
 But, if we have the edge $xy$, then $|S_y \cap  R_1| =1$ as otherwise the dominating set $\{c_0,y\}$ would contradict the maximality of $D$. Hence, by the triangulation of $\overrightarrow{G}$, 
 and in order to have weak directed distance at most 2 from the vertices of $S_x \cup \{x\}$,
  each vertex of $S_y \cap R_0$ will be adjacent either to $c_3$ or to $x_1$. This will create a dominating set
  $\{x_1,y\}$ or $\{c_3,y\}$ that contradicts the maximality of $D$. Hence, we do not have the edge $xy$ (not even in other regions) and we thus have the edge $c_2c_3$.
   
For $|S_x| \leq 2$, the vertices of $S_y \cap R_0$ will be adjacent to either $c_3$, $c_0$ or $x_1$ in order to 
have weak directed distance at most 2 from $x$. 
But then, the triangulation of $\overrightarrow{G}$ will force at least two vertices of $S_y \cap R_0$ to be  common neighbors of   $c_3$ and $x_1$,
 or to have the edge $c_0c_3$. 
It is not difficult to check, casewise (drawing a picture for individual cases will help in understanding the scenario), that
one of the sets $\{c_0,y\}$, $\{c_3,y\}$ or $\{x_1,y\}$ would then be a dominating set contradicting the maximality of $D$.

For $|C| = 5$, by Lemma~\ref{adjacent}, each vertex of $S_y \cap R_i$ must disagree with $c_{i+2}$ on $y$. If the vertices of  
$S_y \cap R_0$ and the vertices of $S_y \cap R_1$ agree with  each other on $y$, then they must disagree with each other on $c_0$,
 which implies $|S_y \cap R_i| \leq 3$ for all $i \in \{0,1 \}$. 
 If the vertices of  $S_y \cap R_0$ and the vertices of $S_y \cap R_1$
  disagree with  each other on $y$, then the vertices of $S_y \cap R_i$  must agree with  $c_{3-i}$ on $y$. 
  In that case, by Lemma~\ref{connection}, each
vertex of $S_y \cap R_i$  must be connected to  $c_{3-i}$ by a 2-dipath through $c_{4-3i}$, which 
implies $|S_y \cap R_i| \leq 3$ 
for all $i \in \{0,1 \}$. 

Assume $|S_y \cap R_0| = 3$ and $|S_y \cap R_1| = 3$.
Then, each vertex of $S_y \cap R_i$ must disagree with both $c_{i+2}$ and $c_{i+3}$ on $y$. 
This would imply that the vertices of  $S_y \cap R_0$ and the vertices of $S_y \cap R_1$
  disagree with  each other on $c_0$. Now, there would be no way to have weak directed distance at most 2 between a vertex of $S_x$ and all the six vertices of $S_y$. 
  
  Hence we must have $|S_y| \leq 5$.
Then, by equation~(\ref{order}), we have
\begin{align} \nonumber
15 \leq | \overrightarrow{G}| \leq 2+5+(2+5) = 14.\nonumber
\end{align}
This is a contradiction, and this concludes this particular subcase.

\item Assume the three non-empty sets are $S_x \cap R_1$, $S_y \cap R_0$ and $S_y \cap R_2$ (only possible for $|C| \geq 3$).  
By Lemma~\ref{connection}, we have  $S_x = \{x_1 \}$ and the fact that each vertex of $S_y \cap R_i$ disagrees with $c_{i^2/4}$ on $x_1$ for $i \in \{0,2 \}$.  
Moreover, the triangulation of $\overrightarrow{G}$ implies the edges $x_1c_0$, $x_1c_1$, $ c_{k-1}c_0$, $ c_{0}c_1$ and $ c_{1}c_2$.

 For $|C|=3$, $\{c_0,c_1 \}$ is a dominating set with at least four common neighbors $\{x,y,c_2,x_1 \}$, contradicting the maximality of $D$. 
 For $|C| =4,5$, every vertex of $S_y \cap R_0$ disagrees with every vertex of $S_y \cap R_2$ on $y$. 
 Hence, by Lemma~\ref{3same}, we have  $|S_y \cap R_i| \leq 3 $ for all $i \in \{0,2 \}$. 
 By equation~(\ref{order}), we then have 
 \begin{align}\nonumber
 15 \leq | \overrightarrow{G}| &= |D| + |C| + |S| \\ \nonumber
 &\leq [2+5+ (1+3+3)] = 14.\nonumber
 \end{align}
 This is a contradiction. 
 \end{enumerate}
 
\medskip

\noindent\textbf{\textit{Step 4.}} Hence,  at most two sets  out of the $2k$ sets $S_t \cap R_i $ can be non-empty, 
where $t \in \{x,y \}$ and  $i \in \{0,1, \dots, k-1 \}$.

Assume that exactly two sets out of the  sets $S_t \cap R_i $ are non-empty, where $t \in \{x,y \}$ and  $i \in \{0, \dots, k-1 \}$, yet there are two non-empty regions. 
Without loss of generality, assume that the two non-empty sets are $S_x \cap R_0$ and $S_y \cap R_1$.  

The triangulation of $\overrightarrow{G}$ would  force $x$ and $y$ to have a common neighbor other than $c_0$ and $c_1$ for $|C|=2$ which is a contradiction. 
For $|C|=3,4,5$ the triangulation of $\overrightarrow{G}$ forces the edges $c_{k-1}c_0$ and $c_0c_1$. By Lemma~\ref{connection}, we know that each vertex of $S$ is adjacent 
 to $c_0$. By equation~(\ref{order}), we have
 \begin{align}\nonumber
 |S| \geq (15-2-5) = 8.\nonumber
 \end{align}
  Hence, without loss of generality, we may assume $|S_x| \geq 4$. But then 
 $\{c_0,x \}$ would be a dominating set with at least six common neighbors $S_x \cup \{c_{k-1},c_1 \}$, contradicting the maximality of $D$. 

 \medskip

 This concludes the proof.
   \end{proof}
 
 The lemma proved above was one of the key steps to prove the theorem. Now we will 
 improve the lower bound on $|C|$.

 \begin{lemma}\label{ge6}
$| C | \geq 6$.
\end{lemma}

\begin{proof}
For $ | C | =2,3,4,5$, without loss of generality by Lemma~\ref{1region}, we may assume $R_1$ to be the only non-empty region.  
The triangulation of $\overrightarrow{G}$ will then force the  
configuration depicted in Figure~\ref{single} as a subgraph of $und( \overrightarrow{G})$, 
where 
$C = \{ c_o, \dots, c_{k-1} \}$, $S_x = \{x_1, \dots, x_{n_x} \}$ and $S_y = \{y_1, \dots, y_{n_y} \}$. 
Without loss of generality, we may assume
\begin{align}\nonumber
| S_y |= n_y \geq n_x =| S_x |.
\end{align}
Then, by equation~(\ref{order}), we have
\begin{equation}\label{choto eqn}
n_y = | S_y | \geq (15-2-|C| - |S_x|) = 13 - |C| - |S_x|.
\end{equation}

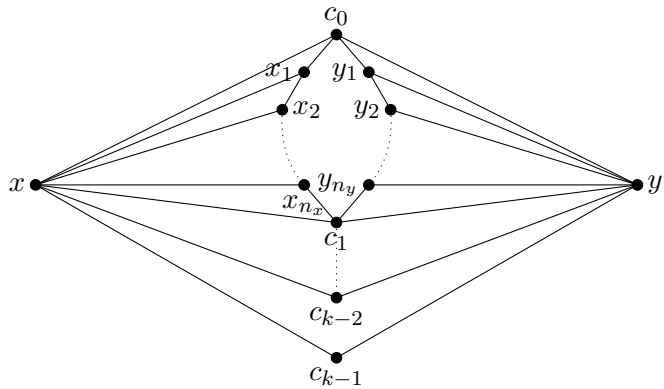
\begin{figure}

\centering
\begin{tikzpicture}

\filldraw [black] (0,2) circle (2pt) {node[left]{$x$}};
\filldraw [black] (8,2) circle (2pt) {node[right]{$y$}};

\filldraw [black] (4,4) circle (2pt) {node[above]{$c_0$}};

\filldraw [black] (4,1.5) circle (2pt) {node[below]{$c_1$}};

\filldraw [black] (4,.5) circle (2pt) {node[below]{$c_{k-2}$}};

\filldraw [black] (4,-.3) circle (2pt) {node[below]{$c_{k-1}$}};


\draw[dotted] (3.28,3) .. controls (3.25,2.7) and (3.3,2.3) .. (3.57,2);

\draw[dotted] (4.72,3) .. controls (4.75,2.7) and (4.7,2.3) .. (4.43,2);


\filldraw [black] (3.57,3.5) circle (2pt) {node[left]{$x_1$}};
\filldraw [black] (3.28,3) circle (2pt) {node[right]{$x_2$}};
\filldraw [black] (3.57,2) circle (2pt) {node[below]{$x_{n_x}$}};

\filldraw [black] (4.43,3.5) circle (2pt) {node[left]{$y_1$}};
\filldraw [black] (4.72,3) circle (2pt) {node[left]{$y_2$}};
\filldraw [black] (4.43,2) circle (2pt) {node[left]{$y_{n_y}$}};

\draw[-] (4,4) -- (3.57,3.5);
\draw[-] (3.28,3) -- (3.57,3.5);
\draw[-] (3.57,2) -- (4,1.5);

\draw[-] (4,4) -- (4.43,3.5);
\draw[-] (4.43,3.5) -- (4.72,3);
\draw[-] (4.43,2) -- (4,1.5);

\draw[-] (0,2) -- (3.57,3.5);
\draw[-] (0,2) -- (3.28,3);
\draw[-] (0,2) -- (3.57,2);

\draw[-] (8,2) -- (4.43,3.5);
\draw[-] (8,2) -- (4.72,3);
\draw[-] (8,2) -- (4.43,2);


\draw[-] (0,2) -- (4,-.3);

\draw[-] (0,2) -- (4,4);
\draw[-] (0,2) -- (4,1.5);

\draw[-] (0,2) -- (4,.5);

\draw[-] (8,2) -- (4,-.3);

\draw[-] (8,2) -- (4,4);
\draw[-] (8,2) -- (4,1.5);

\draw[-] (8,2) -- (4,.5);

\draw[dotted] (4,1.5) -- (4,.5);

\end{tikzpicture}

\caption{The only non-empty region is $R_1$}\label{single}

\end{figure}

\medskip

 First of all, assume $n_x = 0$. Then $x$ is not adjacent to $y$, as otherwise $y$ would dominate the whole graph. 
 So we have the edges $c_0c_1$, $c_1c_2$, $\dots$, $c_{k-1}c_0$ by the triangulation of $\overrightarrow{G}$.
Then, by  equation~\ref{choto eqn}, we have
\begin{align}\nonumber
|S_y| \geq 13 - 5 = 8.
\end{align}

 Now, to have $\overline{d}(x,y_i) \leq 2$, every $y_i$ must be connected to $x$ by a 2-dipath with internal vertex either $c_0$ or $c_1$. 
 Hence, at least four vertices of $S_y$ must be adjacent to either $c_0$ or $c_1$. 
 Note that $c_0$ is also adjacent to $c_{k-1},c_1$ and that $c_1$ is also adjacent to 
$c_0,c_2$. But then, the dominating set $\{c_0,y\}$ or $\{c_1,y\}$ will contradict the maximality of 
$D$. 
Hence $n_x \geq 1$.

\medskip

The proof will now directly follow from the  four claims below.

\medskip

\noindent\textbf{\textit{Claim 1:}} $|C| = 5$ is not possible.

\medskip

\noindent\textbf{\textit{Proof of claim 1:}} Assume that $|C| = 5$. Then, by  equation~\ref{choto eqn}, we have
\begin{align}\nonumber
|S_y| \geq 13 - 5 - n_x = 8 - n_x.
\end{align}
Therefore, as $n_y \geq n_x$, we have $n_y \geq 4$. 
Now, every vertex of 
$S_y$ disagrees with $c_3$ on $y$. They  must also disagree with $y$ on $c_2$,  as otherwise all 
of them would be connected to $c_2$ by 2-dipaths with internal vertex $c_1$, which would imply 
$\overline{d}(y_1,y_4) > 2$. 
For similar reasons, the vertices of $S_y$ must disagree with $c_4$ on $y$.

Moreover, the edge $c_0c_1$ does not exist since it would force each vertex of $S_y$ to be
connected to vertices of $S_x$ by 2-dipaths with internal vertex either $c_0$ or $c_1$. 
In fact, for $n_x \geq 2$, as not all vertices of $S_x$ can be adjacent to both $c_0$ and $c_1$, 
every vertex of $S_y$ would be connected to the vertices  of $S_x$ by 2-dipaths with 
internal vertex being exactly one of $c_0, c_1$, thus implying $\overline{d}(y_1,y_{4}) > 2$. 
For $n_x = 1$, as $n_y \geq 7$,  at least four vertices of 
$S_y$ would be connected to the vertices  of $S_x$ by 2-dipaths with 
internal vertex being exactly one of $c_0, c_1$, implying $\overline{d}(y_i,y_{i+3}) > 2$ for 
some $i \in \{1, 2, \dots, n_y\}$.
Hence, the edge $c_0c_1$ does not exist. 

Also, if we have the edge $y_1y_4$ and, without loss of generality,  the edge
$y_1y_3$ by the triangulation of $\overrightarrow{G}$, then every vertex of $S_x$ must be connected to $y_2$ by 2-dipaths 
with internal vertex $y_1$. 
In this case, $\{y_1,y\}$ is a dominating set with at least $n_y$ common neighbors ($c_0$ and $n_y-1$ common neighbors from $S_y$). 
Hence, to avoid a contradiction with the maximality of $D$, we must have $n_y \leq 5$. 
We must also have $n_x \geq 3$. 
But then, as every vertex of $S_x$ agree with each other on $y_1$ 
and on $x$ (as they all disagree with $c_3$ on $x$),
they must all disagree with $c_1$ and $c_4$ on $x$ to have weak directed distance at most 2 with them.
Also the vertices of $S_y$ must disagree with $c_1$ and $c_4$ on $y$ to have directed distance at most 2 with them.
So, $c_1$ and $c_4$ agrees with each other on both $x$ and $y$,
Therefore, to have weak directed distance at most 2 between $c_1$ and 
$c_4$ we must have the 2-dipath connecting $c_4$ and $c_1$ with internal vertex $c_0$.
But this is a contradiction as we can not have the edge $c_0c_1$.

Similarly, we cannot  have the edge $y_1y_3$ also.
Therefore, $y_1$ and $y_4$ must be connected by a 2-dipath with an internal vertex $x_j$ from $S_x$
for some $j\in \{1,2,..,n_x\}$. 
As we cannot have the edge $y_1y_4$, 
this  implies that every vertex of $S \setminus \{x_j\}$ is adjacent to $x_j$ to be at weak directed distance at most 2 from each other. 
We can then reach a contradiction, exactly as in the case described in the paragraph above. 

This proves the claim. \hfill $\lozenge$

\medskip

\noindent\textbf{\textit{Claim 2:}} $|C| = 4$ is not possible.

\medskip

\noindent\textbf{\textit{Proof of claim 2:}} Assume that $|C| = 4$. Then, by  equation~\ref{choto eqn}, we have
\begin{align}\nonumber
|S_y| \geq 13 - 4 - n_x = 9 - n_x.
\end{align}
Therefore, as $n_y \geq n_x$, we have $n_y \geq 5$.

We now show that every vertex of $S_y$ disagrees with $c_2$ and $c_3$ on $y$. 
First note that
no vertex can agree with both $c_2$ and $c_3$ on $y$ as otherwise it would be adjacent to 
both $c_0$ and $c_1$, which is impossible since 
$n_y \geq 5$. 
So,  if the claim is not true, then some vertices of 
$S_y$ will agree with $c_2$ on $y$ and the other vertices of $S_y$ will agree with $c_3$ on $y$.

Also at most three vertices of  $S_y$ can agree with $c_2$ (or $c_3$) on $y$. 
So, $n_y \leq 6$. Hence, $n_x \geq 3$. 

Now, three vertices of $S_y$ agree on $y$ with, say, $c_2$. Then they will all disagree with $c_2$ on 
$c_1$ and every vertex (there are at least three such vertices)
of $S_x$ will disagree with those three vertices on $c_1$. 
Then, to have weak directed distance at most 2 between the vertices of $S_x$, the other 
vertices (there are at least two such vertices) of $S_y$ should be adjacent to 
$c_1$, which is not possible as they are already connected to $c_3$ with 2-dipaths with internal vertex $c_0$.

\medskip

The rest of the proof is similar to the proof of Claim 1.  Using similar arguments,
it is possible to show that the edge $c_0c_1$ does not exist, that the edge $y_1y_4$ does not exist and that it is not possible to have a 2-dipath with internal vertex from $S_x$ connecting 
$y_1$ and $y_4$.   \hfill $\lozenge$

\medskip 

\noindent\textbf{\textit{Claim 3:}} $|C| = 3$ is not possible.

\medskip

\noindent\textbf{\textit{Proof of claim 3:}} Assume that $|C| = 3$. Then, by  equation~\ref{choto eqn}, we have
\begin{align}\nonumber
|S_y| \geq 13 - 3 - n_x = 10 - n_x.
\end{align}
Therefore, as $n_y \geq n_x$, we have $n_y \geq 5$.

First note that it is not possible to have the edge $c_0c_1$, as this will force some three vertices of $S_y$ to be connected to vertices of $S_x$ by 2-dipaths with internal vertex $c_0$ (or $c_1$),
making $\{c_0,y\}$ (or $\{c_1,y\}$) a dominating set that contradicts the maximality of $D$.

For $n_y \geq 7$, there are at least four vertices in $S_y$ that agree with each other on $y$. 
We need to have weak directed distance at most 2 between them. Let those four vertices be $y_i,y_j,y_k,y_l$ with $i >j>k>l$. 

Now, assume that we have the edge $y_iy_l$. Then, every vertex of $S_x$ will be adjacent 
to either $y_i$ or $y_l$. Without loss of generality, assume that every vertex of $S_x$ is adjacent 
to $y_i$. But then, $\{y_i,y\}$ would be a dominating set with at least four common neighbors ,
contradicting the maximality of $D$. Hence $n_y \leq 6$ and, therefore, we must have $n_x \geq 4$.

For $n_y = 5,6$, one can show that these cases are not possible without creating a dominating set that contradicts the maximality of $D$. 
If one just tries to have weak directed distance at most 2 between the vertices of $S$, the proof will follow. The proof of this part is also similar to the ones done before and, though a bit tedious, is not difficult to check.  \hfill $\lozenge$

\medskip

\noindent\textbf{\textit{Claim 4:}} $|C| = 2$ is not possible.

\medskip

\noindent\textbf{\textit{Proof of claim 4:}} Assume that $|C| = 2$. Then, by  equation~\ref{choto eqn}, we have
\begin{align}\nonumber
|S_y| \geq 13 - 2 - n_x = 11 - n_x.
\end{align}
Therefore, as $n_y \geq n_x$, we have $n_y \geq 6$.

This is actually the easiest of the four claims. The case $n_y \geq 7$ can 
be argued as in the previous proof. For $n_y = 6$, we must have $n_x \geq 5$. 
If one just tries to have directed distance at most 2 between the vertices of $S$, the proof will follow. The proof of this part is also similar to the ones done before and, though a bit tedious, is again not difficult to check. \hfill $\lozenge$

\medskip

This completes the proof of the lemma.
  \end{proof}


\medskip

Up to now, we have proved that the value of $|C|$ is at least 6. This is an answer to our question ``how small $|C|$ can be?''. 
We will now consider the question ``How big $|C|$ can be?'' and try to provide upper bounds for the value of $|C|$. The following lemma will help us to do so.

\begin{lemma}\label{helper}
If $|C| \geq 6$, then the following holds:
\begin{enumerate}[{\rm (a)}]
\item $|C^{\alpha \beta}| \leq 3$,  $|C_t^{\alpha}| \leq 6$, $|C| \leq 12$. Moreover, 
if $|C^{\alpha \beta}| = 3$, then $ \overrightarrow{G}[C^{\alpha \beta}]$ is a 2-dipath.
\item $|C_t^{\alpha}| \geq 5$ (respectively $4,3,2,1,0$) implies $|S^{\alpha}_t| \leq 0$ (respectively $1,3,4,5,6$). 
\end{enumerate}
\end{lemma} 

\begin{proof}

\noindent{(a)} If  $|C^{\alpha \beta}| \geq 4$, then there will be two vertices $u,v \in C^{\alpha \beta}$ with $d(u,v)>2$, which is a contradiction. 
Hence we have the first inequality, which implies the other two. 

If 
$|C^{\alpha \beta}| =3$, then the only way to connect the 
two non-adjacent vertices $u,v$ of $C^{\alpha \beta}$ is 
to connect them with a 2-dipath through the other vertex (other than $u,v$) of $C^{\alpha \beta}$.
 
\medskip 
 
\noindent{(b)}  Lemma~\ref{adjacent}(b) implies that if all the elements of $C_t^{\alpha}$ do not belong to  the set of four boundary  points
  of any three consecutive regions (like $R,R^1,R^2$ in Lemma~\ref{connection}), then $|S^{\alpha}_t| = 0$.
  Hence, we have $|C_t^{\alpha}| \geq 5$ implies $|S^{\alpha}_t| \leq 0$.

\medskip   
   
 By Lemma~\ref{connection},   if all the elements of $C_t^{\alpha}$   belong to the set of four boundary  points
  $c^1,c^2,\overline{c^1},\overline{c^2}$ of three consecutive regions $R,R^1,R^2$ (like in Lemma~\ref{connection}) and contains both $\overline{c^1},\overline{c^2}$, then $|S^{\alpha}_t| \leq 1$. Also $S^{\alpha}_t \subseteq R$ by Lemma~\ref{connection}. Hence we have  
  \begin{align}\nonumber
  |C_t^{\alpha}| \geq 4 \text{ implies } |S^{\alpha}_t| \leq 1.
  \end{align}

  Assume now that  all the elements of $C_t^{\alpha}$ belong to the set of three boundary  points
  $c^1,c^2,\overline{c^1}$ of two adjacent regions $R,R^1$ (like in Lemma~\ref{connection}) and contain both $\overline{c^1},c^2$. 
  Then, by Lemma~\ref{adjacent}, $v\in S^\alpha_t$ implies $v$ is in $R$ or $R^1$. 
  
  Now, if both $S^\alpha_t \cap R$ and 
  $S^\alpha_t \cap R^1$ are non-empty, then each vertex of  $(S^\alpha_t \cap R) \cup \{ c^2 \}$ disagrees with each vertex of $(S^\alpha_t \cap R^1) \cup \{ \overline{c^1} \}$
 on $c^1$ (by Lemma~\ref{connection}). 
 
 Hence, by Lemma~\ref{3same}, we have
 \begin{align}\nonumber
 |(S^\alpha_t \cap R) \cup \{ \overline{c^1} \}| ,|(S^\alpha_t \cap R^1) \cup \{ c^2 \}| \leq 3.
 \end{align}
 This clearly implies
 \begin{align}\nonumber
 |S^\alpha_t \cap R|,|S^\alpha_t \cap R^1| \leq 2 \text{ and } |S^\alpha_t| \leq 4.
 \end{align}

 Suppose now that we have $|S^\alpha_t| = 4$ and, hence, also 
 $|S^\alpha_t \cap R|,|S^\alpha_t \cap R^1| = 2$. 
 Then, $S_{t'} = \emptyset $  as the only way for a vertex of $S_{t'}$ to be at 
weak directed  distance at most 2 from every vertex of $S_t$ is by being connected by a 2-dipath with internal vertex $c_1$, which is impossible as the vertices of 
$S^\alpha_t \cap R$ disagree  with the vertices of $S^\alpha_t \cap R^1$ on $c_1$. 

In fact, for the same reason, it is impossible to have weak directed distance at most 2 between all the vertices of $S_t$ and  $t'$ unless we have the edge $tt'$ (that is the edge $xy$). But then, the edge 
$tt'$ makes $t$ a vertex that dominates the whole graph, contradicting the 
domination number of the graph being 2. Therefore, it is not possible to have $|S^\alpha_t| = 4$. 
Hence, we have $|S^\alpha_t| \leq 3$ in this case.

Also, if one of $S^\alpha_t \cap R$ and 
  $S^\alpha_t \cap R^1$ is empty then we must have  $|S^\alpha_t| \leq 3$ by Lemmas~\ref{connection} and \ref{3same}. 
  Hence, we have 
  \begin{align}\nonumber
  |C_t^{\alpha}| \geq 3 \text{ implies } |S^{\alpha}_t| \leq 3.
  \end{align}

 Let $R,R^1,R^2,c^1,c^2,\overline{c^1},\overline{c^2}$ be as in Lemma~\ref{connection} and assume $C_t^{\alpha} = \{c^1,c^2 \}$.  
 By Lemma~\ref{adjacent}, $v\in S^\alpha_t$ implies $v$ is in $R$, $R^1$ or $R^2$, and also that
 both $S^\alpha_t \cap R^1$ and $S^\alpha_t \cap R^2$ cannot be non-empty. 
 Hence, without loss of generality, assume  $S^\alpha_t \cap R^2 = \emptyset$.

 By Lemma~\ref{connection}, the vertices of  
 $S^\alpha_t \cap R^1$  disagree with the vertices of $(S^\alpha_t \cap R) \cup \{c^2 \}$ on $c^1$. 
  Hence, by Lemma~\ref{3same}, we have
  \begin{align}\nonumber
  |S^\alpha_t \cap R^1|,|(S^\alpha_t \cap R) \cup \{c^2 \}| \leq 3.
\end{align}
This implies   $|S^\alpha_t| \leq 5$.

  Now, if $S^\alpha_t \cap R^1 = \emptyset$ then $S_t^\alpha = S^\alpha_t \cap R$. 
  Let $|S^\alpha_t \cap R| \geq 6$ and consider the induced graph $ \overrightarrow{O} = \overrightarrow{G}[(S \cap R) \cup \{c^1,c^2 \}]$. 
  In this graph, the vertices of $(S^\alpha_t \cap R) \cup \{c^1,c^2 \}$ are  at weak directed distance at most 2 from each other. 
  Hence, $\chi_o( \overrightarrow{O}) \geq 8$. But this is a contradiction since $ \overrightarrow{O}$ is an outerplanar graph and every outerplanar graph has an oriented $7$-coloring~\cite{ocoloring}. 
  Hence,
  \begin{align}\nonumber
  |C_t^{\alpha}| \geq 2 \text{ implies } |S^{\alpha}_t| \leq 5.
  \end{align}
  Suppose now that we have $|S^{\alpha}_t| = 5$. Then we must have 
  $S_{t'} = \emptyset$ as otherwise it is not possible to have weak directed distance at most 2 between the vertices of $S$. 
  
  We also do not have the edge $xy$ as it would contradict the 
  domination number of the graph being 2 ($t$ will dominate the graph).
  So, by the triangulation of $\overrightarrow{G}$, we  have the edges $c^1c^2$ and  $c^{\overline{1}}c^1$. 
   Hence, each vertex of $S_t$ must 
  be connected to $t'$ with a 2-dipath with internal vertices 
  from $\{c^{\overline{1}}, c^1, c^2 \}$. But then, it will not be 
  possible to have weak directed distance at most 2 between the five vertices of $S_t$. 
  
  Hence, 
  \begin{align}\nonumber
  |C_t^{\alpha}| \geq 2 \text{ implies } |S^{\alpha}_t| \leq 4.
  \end{align}

 In general, $S^\alpha_t$ is contained in two distinct adjacent regions by Lemma~\ref{adjacent}. 
 Without loss of generality, assume $S^\alpha_t \subseteq R_1 \cup R_2$. 
 If both $S^\alpha_t \cap R_1$ and $S^\alpha_t \cap R_2$ are non-empty then, by Lemma~\ref{connection}, we know that 
 the vertices of $S^\alpha_t \cap R_1$  disagree with the vertices of $S^\alpha_t \cap R_2$ on $c_1$. 
 Hence, $|S^\alpha_t \cap R_1|,|S^\alpha_t \cap R_2| \leq 3$, which implies 
 $|S^\alpha_t| \leq 6$.

 Assume now that only one of the two sets $S^\alpha_t \cap R_1$ and  $S^\alpha_t \cap R_2$ is non-empty. 
 Without loss of generality, assume $S^\alpha_t \cap R_1 \neq \emptyset$. 
 If  $c_0,c_1 \notin C^\alpha_t $ and $|C^\alpha_t| = 1$ then we have  
  $|S^\alpha_t \cap R_1| \leq 3$ by Lemmas~\ref{connection} and~\ref{3same}. 
  In  the  induced outerplanar graph 
  $ \overrightarrow{O} = \overrightarrow{G}[(S \cap R_1) \cup \{c_1,c_2 \}]$, the   vertices of $S^\alpha_t  \cup (c^\alpha_t \cap \{c_1,c_2 \})$ are  at weak directed distance at most 2
  from each other.

  Hence, $7 \geq \chi_o( \overrightarrow{O}) \geq |S^\alpha_t  \cup (c^\alpha_t \cap \{c_1,c_2 \})|$.
  Therefore,
  \begin{align}\nonumber
  |C_t^{\alpha}| \geq 1 \text{ (respectively  $0$)   implies } |S^{\alpha}_t| \leq 6 \text{ (respectively $7$).}
  \end{align}
Now, when both the equalities hold,  we must have $S_{t'} = \emptyset $ as otherwise 
$C \cup S_t \cup S_{t'}$ would contain an oriented outerplanar graph with oriented chromatic number at least 8, which is not possible, in order to have all the vertices 
of $S$ at weak directed distance at most 2 from each other. 

Now, $S_{t'} = \emptyset $ would imply that the edge $xy$ is not there, as otherwise $t$ would dominate the whole graph. 
Hence, each vertex of $S^\alpha_t$ must be connected to $t'$ by a 2-dipath with internal vertex $c_i$ for some $i \in \{0,1,2 \}$. 
But this would force  $|S^\alpha_t \cup C^\alpha_t| \leq 6$ as 
otherwise the vertices of 
$S^\alpha_t \cup C^\alpha_t$ would no longer be at weak directed distance at most 2 from each other.

Hence,
  \begin{align}\nonumber
  |C_t^{\alpha}| \geq 1 \text{ (respectively  $0$)   implies } |S^{\alpha}_t| \leq 5 \text{ (respectively $6$),}
  \end{align}
   and we are done.
    \end{proof}

We now prove that the value of $|C|$ can be at most 5, which contradicts our previously proven lower bound on $|C|$. That actually proves Lemma~\ref{lem planar dom 2 implies order 14}. 

\begin{lemma}\label{le5}
$| C | \leq 5$.
\end{lemma}

\begin{proof}
Without loss of generality,  we can suppose $|C_x^\alpha| \geq |C_y^\beta| \geq |C_y^{\overline{\beta}}| \geq |C_x^{\overline{\alpha}}|$ (the last inequality is forced). 
We know that $|C| \leq 12$ and that $|C_x^{\alpha}| \leq 6$ (Lemma~\ref{helper}(a)). 
Therefore, it is enough to show  
that $|S| \leq 12-|C|$ for all possible values of $(|C|,|C^\alpha_x|,|C^\beta_y|)$,  since it contradicts~(\ref{order}). 

\medskip

For $(|C|,|C^\alpha_x|,|C^\beta_y|) = (12,6,6)$, $(11,6,6)$, $(10,6,6)$, $(10,6,5)$, $(10,5,5)$, $(9,5,5)$,  $(8,4,4)$ we have $|S| \leq 12-|C|$, using Lemma~\ref{helper}(b). 

For $(|C|,|C^\alpha_x|,|C^\beta_y|) = (8,6,6)$, $(7,6,6)$, $(7,6,5)$,$(6,6,6)$,  $(6,6,5)$,  $(6,6,4)$,  $(6,5,5)$ we are forced 
to have
\begin{align}\nonumber
|C^{\alpha \beta}| > 3.
\end{align}
This is a contradiction  by Lemma~\ref{helper}(a).

So, $(|C|,|C^\alpha_x|,|C^\beta_y|) \neq (12,6,6)$, $(11,6,6)$, $(10,6,6)$, $(10,6,5)$, $(10,5,5)$, $(9,5,5)$,  $(8,4,4)$, $(8,6,6)$, $(7,6,6)$, $(7,6,5)$,$(6,6,6)$,  $(6,6,5)$,  $(6,6,4)$,  $(6,5,5)$.

We will be done if we prove that $(|C|,|C^\alpha_x|,|C^\beta_y|)$ cannot 
take the other possible values also. That leaves us checking a lot of cases. We will check 
just a few cases and observe that the other cases can be checked using similar arguments.

\medskip

\noindent\textbf{\textit{Case 1:}} Assume $(|C|,|C^\alpha_x|,|C^\beta_y|) = (9,6,6)$.

We are then forced to have
$|C^{\alpha\beta}| = |C^{\alpha \overline{\beta}}| =
|C^{\overline{\alpha} \beta}| = 3$ in order to satisfy the first inequality of Lemma~\ref{helper}(a). 
So, $\overrightarrow{G}[C^{\alpha\beta}]$,
$ \overrightarrow{G}[C^{\alpha \overline{\beta}}]$ and $ \overrightarrow{G}[C^{\overline{\alpha} \beta}]$ are  2-dipaths by Lemma~\ref{helper}(a). 
Without loss of generality, we can assume $C^{\alpha \overline{\beta}}=\{c_0,c_1,c_2 \}$ and $C^{\overline{\alpha} \beta} = \{c_3,c_4,c_5 \}$. 
Hence, by  Lemma~\ref{adjacent}, we have
  $u \in R_1 \cup R_2$ and $v \in R_4 \cup R_5$ for any $(u,v)\in S^{\overline{\beta}}_y \times S^{\overline{\alpha}}_x$.  
  Now, by Lemma~\ref{adjacent}, either $S^{\overline{\beta}}_y$ or $S^{\overline{\alpha}}_x$ is empty. 
  Without loss of generality, assume $S^{\overline{\beta}}_y = \emptyset$. Therefore, we have
 $|S| = |S_x| = |S^{\overline{ \alpha}} _x| \leq 3$ 
 (by Lemma~\ref{helper}(b)). So this case is not possible.
 
\medskip 
 
\noindent\textbf{\textit{ Case 2:}} Assume $(|C|,|C^\alpha_x|,|C^\beta_y|) = (7,6,4)$.

So, without loss of generality, we can assume that 
$ \overrightarrow{G}[C^{\alpha \beta}]$ and $ \overrightarrow{G}[C^{\alpha \overline{\beta}}]$  are  2-dipaths, and  $C^{\alpha \beta}=\{c_0,c_1,c_2 \}$, $C^{\alpha \overline{\beta}} = \{c_3,c_4,c_5 \}$ and $C^{\overline{\alpha} \beta} = \{c_6 \}$. 

By Lemma~\ref{helper}, we have $|S_x| \leq 5$ and $|S_y| \leq 3+1 =4$. 
So we are done if either $S_x = \emptyset$ or $S_y = \emptyset$. 

So assume both $S_x$ and $S_y$ are non-empty. 
First assume that  
$S^\beta_y \neq \emptyset$. Then, by Lemma~\ref{adjacent}, we have $S^\beta_y \subseteq R_5$, 
$S^{\overline{\alpha}}_x \subseteq R_5 \cup R_6$ and hence $S^{\overline{\beta}}_y = \emptyset$. 
By Lemma~\ref{connection}, the vertices of $S^\beta_y$ and the vertices of $S^{\overline{\alpha}}_x \cap R_5$ must disagree with $c_6$ on $c_5$ while disagreeing 
with each other on $c_5$, which is not possible. 
Hence, $S^{\overline{\alpha}}_x \cap R_5 = \emptyset$. 
Also, $|S^{\overline{\alpha}}_x \cap R_6| \leq 3$ as they all disagree on $c_5$ with the vertices of $S^\beta_y$. 
Hence, $|S| \leq 4$ when $S^\beta_y \neq \emptyset$.

Now assume $S^\beta_y = \emptyset$ hence $S^{\overline{\beta}}_y \neq \emptyset$. Then, by Lemma~\ref{adjacent}, we have 
$S^{\overline{\beta}}_y \subseteq R_1 \cup R_2$, 
$S^{\overline{\alpha}}_x \subseteq R_0 \cup R_1$ and hence $S^{\beta}_y = \emptyset$. 
Assume $S^{\overline{\beta}}_y \cap R_2 = \emptyset$, as otherwise the vertices of $S^{\overline{\alpha}}_x$ would be adjacent to both $c_0$ and $c_1$ (to be connected to $c_6$ and to
vertices of $S^{\overline{\beta}}_y \cap R_2$ by a 2-dipath), implying $|S^{\overline{\alpha}}_x| \leq 1$, implying $|S| \leq 5$. 
If $S^{\overline{\alpha}}_x \cap R_0 \neq \emptyset$ then 
$|S^{\overline{\beta}}_y \cap R_1|=1$, $|S^{\overline{\alpha}}_y \cap R_1| \leq 1$ 
and $|S^{\overline{\alpha}}_y \cap R_0| \leq 3$, by Lemma~\ref{connection}, and hence $|S| \leq 5$. 
If $S^{\overline{\alpha}}_x \cap R_0 = \emptyset$ then  $|S^{\overline{\beta}}_y \cap R_1| \leq 2$, $|S^{\overline{\alpha}}_y \cap R_1| \leq 3$ and hence $|S| \leq 5$. So also this case is not possible.
 
\medskip 
 
 Similarly one can handle the remaining cases. 
  \end{proof}

From the above lemmas, we get that every planar oclique of order at least 
15 is dominated by a single vertex. Moreover, we also  proved that
a planar oclique dominated by one vertex can have order at most 15. 
Hence, there is no planar oclique of order more than 15. We also proved that every oclique of order 15 must contain 
the planar oclique depicted in Figure~\ref{fig planarmax} as a \ESOK{spanning} subgraph. 

This concludes the proof of Theorem~\ref{unique}. \hfill $\square$

 \section{Proof of Theorem~\ref{orientedplanarabsolute}}\label{sec orientedplanarabsolute}

\noindent{(a)} The proof  directly follows from Theorem~\ref{unique}. 

\medskip

\noindent{(b)} In 1975, Plesn\'ik~\cite{plesnik} characterized and listed all triangle-free planar graphs with diameter 2 (see Theorem~\ref{theorem plesnik}). They are precisely the graphs depicted 
in Figure~\ref{Plesnik}.
Now note that every orientation of those graphs admits a homomorphism to the graphs depicted in Figure~\ref{oPlesnik}, respectively (that is, any oriented graph with underlying graph from the first, second and third family of graphs described in Figure~\ref{Plesnik} admits \ESOK{a homomorphism} to the first, second and third oriented graph depicted in Figure~\ref{oPlesnik}, respectively).

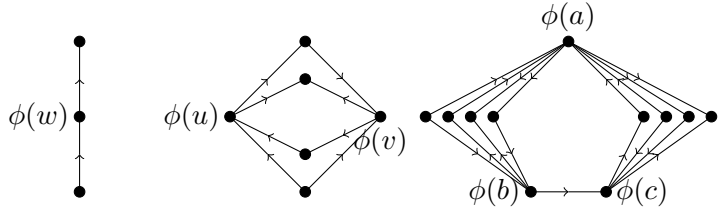
\begin{figure}

\centering
\begin{tikzpicture}

\filldraw [black] (0,0) circle (2pt) {node[left]{}};
\filldraw [black] (0,1) circle (2pt) {node[left]{$\phi(w)$}};
\filldraw [black] (0,2) circle (2pt) {node[right]{}};

\draw[->] (0,0) -- (0,.5);
\draw[-] (0,.5) -- (0,1);

\draw[->] (0,1) -- (0,1.5);
\draw[-] (0,1.5) -- (0,2);

\filldraw [black] (3,0) circle (2pt) {node[left]{}};
\filldraw [black] (3,.5) circle (2pt) {node[right]{}};
\filldraw [black] (3,1.5) circle (2pt) {node[right]{}};
\filldraw [black] (3,2) circle (2pt) {node[left]{}};

\filldraw [black] (2,1) circle (2pt) {node[left]{$\phi(u)$}};
\filldraw [black] (4,1) circle (2pt) {node[below]{$\phi(v)$}};

\draw[->] (2,1) -- (2.5,1.5);
\draw[-] (2.5,1.5) -- (3,2);

\draw[->] (2,1) -- (2.5,1.25);
\draw[-] (2.5,1.25) -- (3,1.5);

\draw[->] (3,.5) -- (2.5,.75);
\draw[-] (2.5,.75) -- (2,1);

\draw[->] (3,0) -- (2.5,.5);
\draw[-] (2.5,.5) -- (2,1);

\draw[->] (3,2) -- (3.5,1.5);
\draw[-] (3.5,1.5) -- (4,1);

\draw[->] (4,1) -- (3.5,1.25);
\draw[-] (3.5,1.25) -- (3,1.5);

\draw[->] (4,1) -- (3.5,.75);
\draw[-] (3.5,.75) -- (3,.5);

\draw[->] (3,0) -- (3.5,.5);
\draw[-] (3.5,.5) -- (4,1);

 \filldraw [black] (6,0) circle (2pt) {node[left]{$\phi(b)$}};
\filldraw [black] (7,0) circle (2pt) {node[right]{$\phi(c)$}};

\filldraw [black] (5.5,1) circle (2pt) {node[left]{}};
\filldraw [black] (7.5,1) circle (2pt) {node[right]{}};

\filldraw [black] (5.2,1) circle (2pt) {node[left]{}};
\filldraw [black] (7.8,1) circle (2pt) {node[right]{}};

\filldraw [black] (4.9,1) circle (2pt) {node[left]{}};
\filldraw [black] (8.1,1) circle (2pt) {node[right]{}};

\filldraw [black] (4.6,1) circle (2pt) {node[left]{}};
\filldraw [black] (8.4,1) circle (2pt) {node[right]{}};

\filldraw [black] (6.5,2) circle (2pt) {node[above]{$\phi(a)$}};

\draw[->] (7.8,1) -- (7.15,1.5);
\draw[-] (7.15,1.5) -- (6.5,2);

\draw[->] (4.9,1) -- (5.7,1.5);
\draw[-] (5.7,1.5) -- (6.5,2);

\draw[->] (4.6,1) -- (5.55,1.5);
\draw[-] (5.55,1.5) -- (6.5,2);

\draw[->] (6.5,2) -- (5.85,1.5);
\draw[-] (5.85,1.5) -- (5.2,1);

\draw[->] (6.5,2) -- (7.3,1.5);
\draw[-] (7.3,1.5) -- (8.1,1);

\draw[->] (6.5,2) -- (7.45,1.5);
\draw[-] (7.45,1.5) -- (8.4,1);

\draw[->] (7.8,1) -- (7.4,.5);
\draw[-] (7.4,.5) -- (7,0);

\draw[->] (8.1,1) -- (7.55,.5);
\draw[-] (7.55,.5) -- (7,0);

\draw[->] (7,0) -- (7.7,.5);
\draw[-] (7.7,.5) -- (8.4,1);

\draw[->] (6,0) -- (5.6,.5);
\draw[-] (5.6,.5) -- (5.2,1);

\draw[->] (6,0) -- (5.45,.5);
\draw[-] (5.45,.5) -- (4.9,1);

\draw[->] (4.6,1) -- (5.3,.5);
\draw[-] (5.3,.5) -- (6,0);


\draw[->] (6,0) -- (6.5,0);
\draw[-] (6.5,0) -- (7,0);

\draw[->] (7,0) -- (7.25,.5);
\draw[-] (7.25,.5) -- (7.5,1);

\draw[->] (7.5,1) -- (7,1.5);
\draw[-] (7,1.5) -- (6.5,2);

\draw[->] (6.5,2) -- (6,1.5);
\draw[-] (6,1.5) -- (5.5,1);

\draw[->] (5.5,1) -- (5.75,.5);
\draw[-] (5.75,.5) -- (6,0);

\end{tikzpicture}

\caption{Planar \ESOK{targets} with girth at least 4}~\label{oPlesnik}

\end{figure}	

To prove the homomorphisms,  we map the vertices $w,u,v$ and $a$ from Figure~\ref{Plesnik} to the corresponding vertices 
$\phi(w),\phi(u),\phi(v)$ and $\phi(a)$ in Figure~\ref{oPlesnik}, respectively. The vertices $b$ and $c$ are
 mapped to 
the vertices $\phi(b)$ (or $\phi(c)$) and $\phi(c)$ (or $\phi(b)$) depending on the orientation of the edge $bc$. 
Without loss of generality, we can assume the edge to be oriented as $\overrightarrow{bc}$ and assume that the vertices $b$ and $c$ map to the vertices 
$\phi(b)$ and $\phi(c)$, respectively. 

Now, to complete the first homomorphism, map the vertices of $N^\alpha(w)$ to the unique vertex in $N^\alpha(\phi(w))$ for 
$\alpha \in \{+,-\}$.

To complete the second homomorphism,  map the vertices of $N^\alpha(u) \cap N^\beta(u)$ to the unique vertex in $N^\alpha(\phi(u)) \cap N^\beta(\phi(v))$ for 
$\alpha, \beta \in \{+,-\}$.

To complete the third homomorphism,  map the vertices of $N^\alpha(a) \cap N^\beta(t)$ to the unique vertex in $N^\alpha(\phi(a)) \cap N^\beta(\phi(t))$ for 
$\alpha, \beta \in \{+,-\}$ and $t \in \{b,c\}$.

Now, note that the first two oriented graphs depicted in Figure~\ref{oPlesnik} are ocliques of order 3 and 6, respectively, while the third graph is not an 
oclique but clearly has \ESOK{absolute oriented clique number} 5. 

Hence, there is no triangle-free planar oclique of order more than 6. 
Also, the only example of a trianlge-free oclique of order 6 is the second graph depicted in Figure~\ref{oPlesnik}.

\medskip

\noindent{(c)} From the proof above, we know that there is no triangle-free planar oclique of order more than 6, and the only example of a triangle-free oclique of order 6 is the second graph depicted in Figure~\ref{oPlesnik}, which is a graph with girth 4. 
Hence, there is no planar oclique with girth at least 5 on more than 5 vertices, while the directed cycle of length 5 is clearly a planar oclique with girth 5.

\medskip

\noindent{(d)} The 2-dipath is an oclique of order 3. From Plesn\'ik's characterization, the rest of the proof follows easily.  \hfill $\square$

\section{Conclusion}\label{sec conclusion}

In this paper we proved three main results regarding the order of planar ocliques, that is oriented planar graphs with 
\textit{weak directed diameter} (that is, the maximum weak directed distance between two vertices of an oriented graph) at most two.
We provided an exhaustive list of spanning subgraphs of outerplanar graphs that admits an orientation with weak directed diameter  at most two.
Now the question is, can a similar result be proved for planar graphs also?

\begin{question}
Characterize the set $L$ of graphs such that a planar graph can be oriented as an oclique if and only if it contains one of the graphs from $L$ as a spanning subgraph.
\end{question}

We partially answer the question by proving that every planar oclique of order 15 must contain a particular oclique as a spanning subgraph. 
As the order of a planar oclique  can at most be 15, a similar study for planar ocliques of order less than 15 will answer the question. 
We also proved tight upper bounds for the order of planar ocliques of girth at least $k$ for all $k \geq 4$.

We defined the parameter oriented relative oclique number and used it for proving Theorem~\ref{unique}. 
Determining oriented relative clique number for different families of graphs, such as the family of planar graphs, seems to be an interesting 
direction of research.

\bibliographystyle{abbrv}
\bibliography{NSS14}

\end{document}